\documentclass[a4paper, reqno]{amsart}
\usepackage{amsthm}
\usepackage{amsmath}
\usepackage[hmarginratio={1:1},vmarginratio={1:1}]{geometry}
\usepackage{amsfonts}
\usepackage{graphicx}
\usepackage{array}
\usepackage{amssymb}
\usepackage{mathrsfs}

\usepackage{pgf}
\usepackage{tikz}
\usetikzlibrary{arrows}
\usetikzlibrary{cd}
\usetikzlibrary{matrix}
\usetikzlibrary{positioning}
\usepackage{color}
\usepackage{empheq}

\usepackage{amsaddr}

\newtheorem{theorem}{Theorem}

\newtheorem{corollary}{Corollary}
\newtheorem{proposition}{Proposition}
\newtheorem{definition}{Definition}
\newtheorem{lemma}{Lemma}
\newtheorem{example}{Example}
\newtheorem{remark}{Remark}

\usepackage[english]{babel}

\usepackage{hyperref}

\newcommand{\be}{\begin{equation}}
\newcommand{\ee}{\end{equation}}

\newcommand*\widefbox[1]{\fbox{\hspace{2em}#1\hspace{2em}}}

\numberwithin{equation}{section}

\begin{document}

\markboth{A. Bautista, A. Ibort, J. Lafuente}{Classification of Jacobi curves}

\title{On the classification of Jacobi curves and their conformal curvatures}

\author{A. BAUTISTA}

\address{Departamento de An\'alisis Econ\'omico: Econom\'{\i}a Cuantitativa \\ Universidad Aut\'onoma de Madrid \\ C/ Francisco Tom\'as y Valiente 5, 28049 Madrid, Spain.}
\email{alfredo.bautista@uam.es}

\author{A. IBORT}

\address{Departamento de Matem\'aticas, Universidad Carlos III de Madrid \\ Avda. de la
Universidad 30, 28911 Legan\'es, Madrid, Spain, and \\ ICMAT, Instituto de Ciencias Matem\'{a}ticas (CSIC-UAM-UC3M-UCM)\\
C/ Nicol\'as Cabrera, 13-15, 28049, Madrid, Spain.}
\email{albertoi@math.uc3m.es}

\author{J. LAFUENTE}

\address{Departamento de Geometr\'{\i}a y Topolog\'{\i}a \\ Universidad Complutense de
Madrid \\ Avda. Complutense s/n, 28040 Madrid, Spain.} 
\email{jlafuente@mat.ucm.es}

\begin{abstract}
This paper describes the theory of Jacobi curves, a far reaching extension of the spaces of Jacobi fields along Riemannian geodesics, developed by Agrachev and Zelenko. Jacobi curves are curves in the Lagrangian Grassmannian of a symplectic space satisfying appropriate regularity conditions.  It is shown that they are fully characterised in terms of a family of conformal symplectic invariant curvatures.  In addition to a new derivation of the Ricci curvature tensor of a Jacobi curve, a Cartan-like theory of Jacobi curves is presented that allows to associate to any admissible Jacobi curve a reduced normal Cartan matrix.   A reconstruction theorem proving that an admissible Jacobi curve is characterised, up to conformal symplectic transformations, by a reduced normal Cartan matrix and  a geometric parametrization is obtained.   The theory of cycles is studied proving that they correspond to flat Jacobi curves.
\end{abstract}

\maketitle




\tableofcontents

\section{Introduction}

A complete characterisation of Jacobi curves will be presented.  Jacobi curves are curves $\Gamma(t)$ of Lagrangian subspaces in a given symplectic space $W$ satisfying appropriate regularity conditions.  A natural extension of Cartan's geometry to curves in Lagrangian Grassmannians will provide the ground to construct a family of curvatures that will classify them.   

The theory of Jacobi curves was initiated by A. Agrachev and I. Zelenko \cite{Ag02,Ag02b} in the context of the theory of control of dynamical systems and their applications (see, for instance, \cite{Ag08} and references therein).   In its original formulation such curves were obtained by means of the flow of a Hamiltonian system on the cotangent bundle $T^* Q$ of a smooth manifold $Q$. In such context a Jacobi curve is the curve in the Lagrangian Grassmannian of the tangent space $T_{(q_0,p_0)}(T^*Q)$ at a fixed point $(q_0,p_0) \in T^*Q$, obtained by pulling back the vertical subspaces of $T(T^*Q)$ along an integral curve of a Hamiltonian system to its initial data $(q_0,p_0)$.   However curves of Lagrangian subspaces appear in a variety of different contexts, for instance, it was recognised early the natural relation between solutions of the Linear-Quadratic regulator problem in optimal control theory, solutions of matrix Ricatti equations and Lagrangian subspaces (see, for instance \cite{Ca80, He80, De03}).   The study of topological properties of curves of Lagrangian subspaces has been also a subject of interest (see, for instance, \cite{Ma15}) or, in a different context, \cite{Ba22a}, for the appearance of Jacobi curves in the study of the geometry of Lorentzian manifolds.   More precisely, if $(M^m, \mathcal{C})$ is a strongly causal conformal Lorentzian spacetime, the manifold $\mathcal{N}$ of its lightrays carries a natural contact structure $\mathcal{H}$ characterized by the fact the each sky $S(x)$ (the set of light rays passing through $x \in M$) is a Legendrian submanifold.  Then for each $\gamma \in \mathcal{N}$, the contact hyperplane $\mathcal{H}_\gamma$ inherits a (conformal) symplectic structure and it makes sense to consider its Lagrangian Grassmannian.  The fact that $\Gamma (t) = T_\gamma (S(\gamma(t)))$ is a Jacobi curve in the Lagrangian Grassmannian of $\mathcal{H}_\gamma$, and that (local) conformal transformations in $M$ preserve the structure $(\mathcal{N}, \mathcal{H})$, will allows us to apply the theory developed in this paper to obtain conformal curvatures over lightrays, that would give rise to conformal invariants of our manifold $(M, \mathcal{C})$.  

One of the main objectives of this work is the construction of geometric invariants associated to such curves that would provide a new and useful insight into the structure of the solutions of matrix Ricatti equation, Hamiltonian systems or any other structure that could be associated to them, like the aforementioned study of causality.  Thus, beyond the interest in control theory, the analysis of Jacobi curves embraces multiple geometrical problems that range from Riemannian and sub-Riemannian geometry (see, for instance, the review \cite{Ag15}) to the geometry of conformal structures  \cite{Ba22}.   

The Jacobi curves considered in this paper are those such that the velocity $\dot{\Gamma}(t)$, interpreted as a symmetric bilinear form on $\Gamma(t)$,  is definite. 
This generic property is preserved under changes of the parameter and the range $[\Gamma]$ of the curve $\Gamma (t)$ will be called an (unparametrized) admissible Jacobi curve.   The conformal symplectic group $CSp$ of linear maps that preserve the symplectic form up to a multiplicative constant, acts on the set of admissible curves.
The main idea developed in this work consists in extending  Cartan's geometry of curves on Euclidean space to analyse Jacobi curves.  For that a new definition of Agrachev and Zelenko's \cite{Ag02} Ricci curvature tensor of a Jacobi curve will be introduced.  Such curvature tensor is obtained after a thorough analysis of the notion of the derivative curve associated to a Jacobi curve, that relies on the local affine structure of the Lagrangian Grassmannian, and the appropriate construction of a geometric arc parameter.   It will be shown that there is a natural geometric parametrization for each Jacobi curve that will allow the construction and definition of the derivative curve and of the Ricci curvature tensor of a Jacobi curve.    

The notion of conformal symplectic curvature invariants of Jacobi curves will be introduced and, using the previously developed tools, a theory of symplectic moving frames will be discussed.    This will allow us to construct a complete, independent and free from integrability conditions family of $(n-1)(n+2)/2$, with $n$ the dimension of the Lagrangian subspaces, conformal symplectic curvature invariants for a given Jacobi curve.  As a consequence any other curvature depends functionally on this family, in particular the higher differential order curvatures obtained in \cite{Ag02,Ag02b}. To conclude the theoretical analysis, a reconstruction theorem will be proved that will show that any admissible Jacobi curve is characterised, up to conformal symplectic transformations, by a reduced normal Cartan matrix  and a geometric arc parameter.   In the particular instance of four-dimensional symplectic spaces, a complete system of conformal invariant curvatures consisting of two independent curvatures, extending the results in \cite{Mu14}, will be obtained.  The reduced normal Cartan matrix of a Jacobi curve is algorithmically computable in the sense that it can be computed using a symbolic manipulation language.  Finally, the general theory of cycles, a special class of Jacobi curves with vanishing curvatures, will be analysed further.

The paper will be organised as follows.   Section \ref{sec:elements} will be devoted to review some basic conformal symplectic notions and notations for the convenience of the reader, most importantly the local affine structure on the Lagrangian Grassmannian of a given symplectic space, structure that will play a relevant role in what follows.  In Sect. \ref{sec:Jacobi} the notion of Jacobi curves, the problem of their classification and the role played by conformal curvatures will be stated.  Section \ref{sec:Ricci} will be devoted to the construction of the Ricci curvature tensor of a Jacobi curve. Such notion will be introduced by the hand of a natural curve associated to any regular Jacobi curve, called its derivative curve.     In Sect. \ref{sec:Cartan} the Cartan geometry of Jacobi curves will be described.  Section \ref{sec:reconstruction} will be devoted to prove the reconstruction and classification theorem for Jacobi curves and, finally, Sect. \ref{sec:cycles} will discuss the theory of cycles, that is, flat Jacobi curves.


\section{Elements of the geometry of the Lagrangian Grassmannian}\label{sec:elements}

As it was indicated in the introduction, the conformal symplectic group is the natural invariance group when dealing with the geometry of Lagrangian subspaces. Thus any attempt of classification of families of Lagrangian subspaces, in particular curves of them, should be invariant with respect to the natural action of such group.   This section will be devoted to establish the basic facts and notations concerning symplectic spaces, the conformal and symplectic groups and the geometry of the collection of all Lagrangian subspaces of a given symplectic space, its Lagrangian Grassmannian.

\subsection{Symplectic preliminaries: conformal symplectic structures}\label{sec:preliminaries}
We will discuss first the emergence of the conformal symplectic group as the natural invariance group of the theory of Lagrangian subspaces and its geometrical structure.

\subsubsection{The symplectic and conformal symplectic groups.}\label{sec:symplectic_group}
Let $(W, \omega)$ be a linear symplectic space, that is, $W$ is a real linear space and $\omega$, called a symplectic form, is a non-degenerate skew-symmetric bilinear form on $W$.  For convenience, we will denote by $(u|v)$ the value of the  symplectic form $\omega$ on the pair of vectors $u,v\in W$, that is, $(u|v) := \omega (u,v)$.    A Lagrangian subspace $\Lambda \subset W$ is a maximal isotropic subspace with respect to the symplectic form $\omega$, that is, a subspace such that $\Lambda = \Lambda^\perp$,  where $\Lambda^\perp$  denotes the subspace of all vectors orthogonal to $\Lambda \subset W$ with respect to the symplectic form $\omega$,  $\Lambda^\perp = \{ u \in W \mid (u| v) = 0, \forall v \in \Lambda\}$.  

The main object of study of this work are curves $\Gamma (t)$ of Lagrangian subspaces, i.e., $\Gamma(t)$ is a Lagrangian subspace of $W$ for each $t$ in some open interval $I \subset \mathbb{R}$ satisfying regularity properties that will be discussed later on (see \S \ref{sec:jacobi_curves}, Def. \ref{def:jacobi}).   
We will denote by $ \mathscr{L}(W)$ the set of all Lagrangian subspaces of $W$ and it will be called the Lagrangian Grassmannian of $W$.  The Lagrangian Grassmannian $\mathscr{L}(W)$ is a closed submanifold of dimension  $n(n+1)/2$  of the Grassmannian manifold $\mathrm{Gr}_n(W)$ consisting of all $n$-dimensional subspaces of the linear space $W$,   $\dim W = 2n$. 

A fundamental observation is that if $\Lambda \subset W$ is a Lagrangian subspace with respect to the symplectic form $\omega$, it is also a Lagrangian subspace for the symplectic form $\lambda \omega$, where $\lambda \neq 0$.  In other words, the Lagrangian Grassmannian $ \mathscr{L}(W)$ is associated to the family $[\omega] = \{ \lambda \omega | \lambda \neq 0 \}$ of symplectic structures on $W$.  We will call such a family $[\omega]$ of (proportional) symplectic structures, a conformal symplectic structure on $W$.  In this sense a proper notation for the Lagrangian Grassmannian would be $ \mathscr{L}(W, [\omega])$, even if we will keep the previous notation for short.

An invertible linear map $c \colon W \to W$ that maps Lagrangian subspaces into Lagrangian subspaces will be called a conformal symplectic map.  A conformal symplectic map $c$ induces a diffeomorphism (denoted with the same symbol) $c \colon  \mathscr{L}(W) \to \mathscr{L}(W)$, $c: \Lambda \mapsto c(\Lambda)$.   Notice that if $c$ is a conformal symplectic map it must transform the conformal class $[\omega]$ into itself, that is, choosing a representative symplectic form $\omega \in [\omega]$, we get $c^* \omega = \lambda_c \omega$, with $\lambda_c \neq 0$, that is: 
\begin{equation}\label{eq:conformal}
(c(u) \mid c(v) ) = \lambda_c (u \mid v) \, , \qquad  \forall u,v\in W \, ,
\end{equation}  
which justifies the given name.   We will denote by $CSp(W,[\omega])$ the Lie group of linear conformal symplectic maps for the conformal structure $[\omega]$. In what follows we will denote the Lie group  $CSp(W,[\omega])$ as $CSp(W)$, or just $CSp$, if there is no risk of confusion, and will be called the conformal symplectic group of the symplectic space $W$.  The Lie algebra $\mathfrak{csp}$ of the Lie group $CSp$ consists of linear maps $A\colon W \to W$ such that there exists $\mu \in \mathbb{R}$ verifying:
\begin{equation}\label{eq:conf_Lie_algebra}
(Au | v) + (u| Av) = \mu (u|v)  \, , \qquad \forall u,v\in W \, .
\end{equation}

It is clear that the map $\lambda \colon CSp \to \mathbb{R}^\times$, $\lambda (c) = \lambda_c$, \textit{cfr.} (\ref{eq:conformal}), associated to the choice of a representative $\omega$ in the conformal symplectic structure of $W$, is a group epimorphism onto the multiplicative group of real numbers, whose kernel is the closed normal subgroup consisting of all linear maps $a \colon W \to W$, such that:
$$
(a(u)|a(v)) = (u|v) \, , \qquad \forall u,v\in W \,  ,
$$ 
called the symplectic group of the symplectic space $W$ and denoted by $Sp(W,\omega)$ (or just $Sp(W)$, or $Sp$, if there is no risk of confusion).     The Lie algebra of the Lie group $Sp$ will be denoted by $\mathfrak{sp}$ and its elements can be identified with skew-linear maps with respect to the bilinear form $(\cdot | \cdot )$, that is,  a linear map $A\colon W \to W$ determines an element in $\mathfrak{sp}$ provided that:
\begin{equation}\label{eq:lie_algebra}
(Au | v) + (u| Av) = 0 \, , \qquad \forall u,v\in W \, .
\end{equation}
From (\ref{eq:lie_algebra}) if follows that the symplectic group $Sp$ has dimension $2n^2 + n$, provided that $\dim W = 2n$, hence $\dim CSp = 2n^2 + n + 1$. 

Given two transversal Lagrangian subspaces $\Lambda$, $\bar{\Lambda} \subset W$, i.e., $\Lambda \cap \bar{\Lambda} = \{ \mathbf{0}\}$, then $W = \Lambda \oplus \bar{\Lambda}$.   We will call such decomposition of $W$ a Lagrangian decomposition.    Choosing a Lagrangian decomposition $W = \Lambda \oplus \bar{\Lambda}$, we can define a cross section of  $\sigma \colon \mathbb{R}^\times \to CSp$, of the short exact sequence $1 \to Sp \to CSp \stackrel{\lambda}{\to} \mathbb{R}^\times \to 1$, given by\footnote{Note that $w(v_1\oplus \bar{v}_1, v_2\oplus\bar{v}_2) = \omega (v_1, \bar{v}_2) + \omega(\bar{v_1}, v_2) = (v_1 \mid \bar{v}_2) +   (\bar{v}_1 \mid v_2)$. Then, $\omega (\sigma(s)(v_1\oplus \bar{v}_1), \sigma(s)(v_2\oplus\bar{v}_2)) = \omega (s v_1\oplus\bar{v}_1, s v_2\oplus\bar{v}_2) = (sv_1\mid   \bar{v}_2) +   (\bar{v}_1 \mid s v_2) = s \omega (v_1\oplus \bar{v}_1, v_2\oplus\bar{v}_2)$, and $\sigma(s) \in CSp$.}: 
\begin{equation}\label{eq:sigma}
\sigma (s) (v \oplus \bar{v}) =  s v \oplus \bar{v} \, , \qquad  \forall v\in \Lambda, \bar{v} \in \bar{\Lambda} \, .
\end{equation}    
Moreover, it is  obvious that the map $\sigma$ is a group homomorphism and $CSp$ becomes the semi-direct product of $Sp$ by $\mathbb{R}^\times$, $CSp = Sp \rtimes \mathbb{R}^\times$. Thus we have the identification $Sp(W) \times \mathbb{R}^\times \cong CSp(W)$, given by $(a,s) \in Sp(W) \times \mathbb{R}^\times \mapsto a \sigma (s) \in CSp (W)$.   We will just write $c = a \sigma (\lambda_c)$, where $a$ is the element of the symplectic group defined as $c \sigma (\lambda_c)^{-1}$.   Note that $\sigma (s) (\Lambda) = \Lambda$.

The definition of the conformal symplectic group makes explicit the natural transitive action of $CSp(W)$ on $\mathscr{L}(W)$ given by $(c, \Lambda) \mapsto c(\Lambda)$, $c \in CSp(W)$ and $\Lambda \in \mathscr{L}(W)$,  that makes it an homogeneous space.  This action restricts to an action of the symplectic group $Sp(W)$ whose orbits coincide.  Indeed, if  $\Lambda$ is a Lagrangian subspace, consider the orbit $\mathcal{O}_\Lambda := CSp (W) \Lambda = \{ c(\Lambda) \mid c \in CSp (W)\}$.   Choose a Lagrangian decomposition of $W$ of the form $W = \Lambda \oplus \bar{\Lambda}$, then using the identification between $CSp$ and $Sp \times \mathbb{R}^\times$,  provided by the cross section $\sigma$ associated to this Lagrangian decomposition, \textit{cfr.} (\ref{eq:sigma}), we get $c (\Lambda) = a \sigma (\lambda_c) (\Lambda) = a (\Lambda)$ and we conclude that the orbit of $\Lambda$ under the conformal symplectic group $CSp$ coincides as sets with the orbit under the symplectic group $Sp$ (although this identification is not canonical and depends on the choice of a cross section $\sigma$).   

 The action of the conformal symplectic group on the Lagrangian Grassmannian is obviously transitive, hence the action of the symplectic group,  which allows to identify $\mathscr{L}(W)$ with the manifold of cosets $\mathscr{L}(W) \cong CSp(W)/H \cong Sp(W)/H_0$, with $H_0 \subset Sp(W)$ the closed subgroup of symplectic transformations fixing the Lagrangian subspace $\Lambda$, $H_0 = \{ a \in Sp(W) \mid a(\Lambda) = \Lambda \}$, and $H = H_0 \rtimes \mathbb{R}^\times \subset CSp$, the corresponding semi-direct product extension subgroup of $CSp$. 
Choosing a metric structure on $W$ compatible with the symplectic form $\omega$, endows $W$ with a K\"ahler structure.    Given a Lagrangian subspace $\Lambda_0$, it is possible to choose the K\"ahler structure in such a way that $\Lambda_0$ becomes a real subspace, then the orbit of the symplectic group becomes the orbit of the unitary group of the K\"ahler structure, which is isomorphic to $U(n)$, and the isotropy group of such action is clearly the orthogonal group of the metric real space $\Lambda_0$, which is isomorphic to $O(n)$, hence the homogeneous space $ \mathscr{L}(W)$ is diffeomorphic to $U(n)/O(n)$ that provides an alternative description of the smooth structure of the Lagrangian Grassmannian.

\subsubsection{Lagrangian subspaces and symplectic basis.} \label{sec:basis} 
It will be helpful to introduce local parametrizations of the Lagrangian Grassmannian in various proofs and computations to follow.  The following paragraphs will be devoted to describe in detail the natural affine atlas of the Lagrangian Grassmannian and a few useful formulas and notations.

Let $\boldsymbol{\epsilon} = (\mathbf{e}, \bar{\mathbf{e}})$ denote a symplectic basis of $W$, that is, $\mathbf{e} = (e_1, \ldots , e_n)$, $\bar{\mathbf{e}} = (\bar{e}_1, \ldots, \bar{e}_n)$, are two systems of $n$ linearly independent vectors in $W$ such that $(e_i|e_j) = (\bar{e}_i | \bar{e}_j) = 0$, $(e_i | \bar{e}_j ) = \delta_{ij}$, for all $i,j = 1, \ldots, n$.  Then we denote by $\Lambda$ ($\bar{\Lambda}$), the Lagrangian subspace generated by $\mathbf{e} = (e_1, \ldots, e_n)$ ($\bar{\mathbf{e}} = (\bar{e}_1, \ldots, \bar{e}_n)$, respec.).   Note that $W = \Lambda \oplus \bar{\Lambda}$.  Conversely, if $\Lambda \in  \mathscr{L}(W)$ is a Lagrangian subspace and $\bar{\Lambda}$ is a Lagrangian subspace transverse to $\Lambda$, that is, $\Lambda \cap \bar{\Lambda} = \mathbf{0}$, then there is a symplectic basis $(\mathbf{e}, \bar{\mathbf{e}})$ such that the system of vectors $\mathbf{e}$ generates $\Lambda$ and $\bar{\mathbf{e}}$ generates $\bar{\Lambda}$.  In such case, we will say that the symplectic basis $(\mathbf{e}, \bar{\mathbf{e}})$  is adapted to the Lagrangian decomposition  $W = \Lambda \oplus \bar{\Lambda}$.  In fact, any one of the systems of vectors, $\mathbf{e} \subset \Lambda$ , $\bar{\mathbf{e}} \subset \overline{\Lambda}$, determines the other (Lemma \ref{lema-base-simplectica} below makes this statement precise).
By the same token we can consider conformal symplectic basis, that is systems of vectors $\boldsymbol{\epsilon} = (\mathbf{e}, \bar{\mathbf{e}})$ such that $(e_i|e_j) = (\bar{e}_i | \bar{e}_j) = 0$, $(e_i | \bar{e}_j ) = \lambda \delta_{ij}$, for all $i,j = 1, \ldots, n$, and $\lambda \neq 0$ called the scaling factor of the basis.   If  $\boldsymbol{\epsilon} = (\mathbf{e}, \bar{\mathbf{e}})$ is a conformal symplectic basis with scaling factor $\lambda$, then $\boldsymbol{\epsilon} = (\mathbf{e}, \bar{\mathbf{e}})$ is a symplectic basis for the symplectic form $\frac{1}{\lambda} \omega \in [\omega]$.  Moreover, given a symplectic basis $\boldsymbol{\epsilon} = (\mathbf{e}, \bar{\mathbf{e}})$, the scaled basis $\boldsymbol{\epsilon}_\lambda = (\mathbf{e}_\lambda, \bar{\mathbf{e}}_\lambda)$, with $\mathbf{e}_\lambda = f(\lambda) \mathbf{e}$, $\bar{\mathbf{e}}_\lambda = g(\lambda)  \bar{\mathbf{e}}$, and $f(\lambda ) g(\lambda) = \lambda$, are new conformal symplectic basis describing the same Lagrangian decomposition of $\Lambda \oplus \overline{\Lambda}$ of $W$.   Thus, when dealing with Lagrangian subspaces we are free to consider either symplectic or conformal symplectic basis to describe them.

Given a Lagrangian subspace $\Lambda\in  \mathscr{L}(W)$ we denote by $\Lambda^\pitchfork$ the open set of Lagrangian subspaces transverse to $\Lambda$, that is:
\begin{equation}\label{eq:transverse}
\Lambda^\pitchfork = \{ \Gamma \in  \mathscr{L}(W) \mid \Gamma \cap \Lambda = \mathbf{0} \} \, .
\end{equation}
The sets $\Lambda^\pitchfork$ can be used to construct a smooth atlas for $\Lambda\in  \mathscr{L}(W)$.   Indeed, consider a Lagrangian decomposition $W = \Lambda \oplus \overline{\Lambda}$, then any Lagrangian subspace $\Gamma$ in $\overline{\Lambda}^\pitchfork$ defines a linear map from $\Lambda$ to $\overline{\Lambda}$  denoted as $\langle \Lambda , \Gamma,  \overline{\Lambda} \rangle$, as  (see Fig. \ref{fig:projectors}): 
\begin{equation}\label{eq:Lambda}
\langle \Lambda , \Gamma,  \overline{\Lambda} \rangle \colon \Lambda \to \overline{\Lambda} \, , \quad \langle \Lambda , \Gamma,  \overline{\Lambda} \rangle (v) = \bar{v}\, , \quad v+\bar{v} \in \Gamma \, , v \in \Lambda\, , \bar{v} \in \overline{\Lambda} \, .
\end{equation}    
If we choose a symplectic basis $(\mathbf{e},\bar{\mathbf{e}})$ adapted to the decomposition $W = \Lambda \oplus \overline{\Lambda}$, then the matrix $S$ associated to the linear map $\langle \Lambda , \Gamma,  \overline{\Lambda} \rangle$, that is, $\langle \Lambda , \Gamma,  \overline{\Lambda} \rangle (e_i) = \sum_j S_{ji} \bar{e}_j$, is symmetric  and the map the assigns to any $n\times n$ symmetric matrix $S$, the Lagrangian subspace $\Gamma_S$ given by:
\begin{equation}\label{eq:eq_Lag}
\Gamma_S = \left\{ (\mathbf{e},\bar{\mathbf{e}}) \left(\begin{array}{c} x \\ Sx \end{array} \right) \mid x \in \mathbb{R}^n \right\} =: \left[\begin{array}{c} I_n \\ S \end{array} \right]_{(\mathbf{e}, \mathbf{\bar{e}})}\, ,
\end{equation}
provides a local chart for $ \mathscr{L}(W)$ on the open set $\overline{\Lambda}^\pitchfork$\footnote{A similar map can be defined using the open set $\Lambda^\pitchfork$ instead, see below this section.}.   
Note that the definition of the subspace $\Gamma_S$ amounts to say that $\bar{x} = S x$ are the equations of $\Gamma_S$ and that the system of vectors $\mathbf{e} + \overline{\mathbf{e}} S$ form a basis of $\Gamma$.
As it will be discussed in the coming section, \S \ref{sec:affine}, such parametrisation is the natural coordinate expression of the canonical affine structure on $\overline{\Lambda}^\pitchfork$.  The initial discussion on the existence of adapted symplectic basis is completed by the following statement.

\begin{lemma}\label{lema-base-simplectica}
Let $\Lambda$ and $\overline{\Lambda}$ be transverse Lagrangian subspaces in $\left(W,\omega\right)$, i.e., $\overline{\Lambda}\cap \Lambda = \{ \mathbf{0}\}$.  Then if $\mathbf{f}=\left( \mathbf{f}_1,\ldots ,\mathbf{f}_m \right)$ is a basis of $\Lambda$, then there is a unique basis $\overline{\mathbf{f}}=\left(\overline{\mathbf{f}}_1,\ldots,\overline{\mathbf{f}}_m \right)$ of $\overline{\Lambda}$ such that $\left(\mathbf{f},\overline{\mathbf{f}}\right)$ is a symplectic basis of $\left(W,\omega\right)$.
\end{lemma}

\begin{proof}   The abstract proof works as follows:  given the basis $\mathbf{f}= (f_1, \ldots, f_n)$ of $\Lambda$, it determines a unique dual basis $\mathbf{f}^* = (f_1^*, \ldots, f_n^*)$ of $\Lambda^*$, $f_i^*(f_j) = \delta_{ij}$. Then, the natural identifications $\overline{\Lambda} \cong W /\Lambda \cong \Lambda^*$, determines the basis $\overline{\mathbf{f}}$ we are looking for.     In spite of this, we will work out the explicit formulas for the vectors in $\overline{\mathbf{f}}$ because they will turn out to be useful later on.
Consider an arbitrary symplectic basis $(\mathbf{e}, \overline{\mathbf{e}})$ such that both $\Lambda$ and $\overline{\Lambda}$ are transverse to the Lagrangian subspace generated by the vectors $\overline{\mathbf{e}}$, then, using the notation introduced above, \textit{cfr.} (\ref{eq:eq_Lag}), we can write:
\[
\Lambda \simeq \left[ \begin{matrix} I_n \\ S \end{matrix} \right]_{(\mathbf{e}, \overline{\mathbf{e}})}\in\mathscr{L}\left(W\right)  \quad \text{ and } \quad \overline{\Lambda} \simeq \left[ \begin{matrix} I_n \\ \overline{S} \end{matrix} \right]_{(\mathbf{e}, \overline{\mathbf{e}})} \in\mathscr{L}\left(W\right)  \,  .
\]
where $S$ and $\overline{S}$ are symmetric matrices. For $v\in \Lambda$ and $\overline{v}\in \overline{\Lambda}$ we can write 
\[
\left\{\begin{matrix} v=\mathbf{e}x+\overline{\mathbf{e}}Sx \\ \overline{v}=\mathbf{e}\overline{x}+\overline{\mathbf{e}}\overline{S}\overline{x} \end{matrix} \right.
\] 
where $x$ and $\overline{x}$ are coordinates of $v$ and $\overline{v}$ with respect to the basis $\mathbf{e} + \overline{\mathbf{e}} S$ and $\mathbf{e} + \overline{\mathbf{e}} \bar{S}$ respectively.  So, we have, Eq. (\ref{eq:coor_sym}):
\begin{align*}
\omega\left(v,\overline{v}\right) & = \omega\left(\mathbf{e}x+\overline{\mathbf{e}}Sx,\mathbf{e}\overline{x}+\overline{\mathbf{e}} \overline{S}\overline{x}\right) = \omega\left(\mathbf{e}x,\overline{\mathbf{e}} \overline{S}\overline{x}\right) + \omega\left(\overline{\mathbf{e}}Sx,\mathbf{e}\overline{x}\right) = \\
& = x^{T}\overline{S}\overline{x} - (Sx)^{T}\overline{x} =  x^{T}\overline{S}\overline{x} - x^{T}S^{T}\overline{x} = x^{T}\left(\overline{S}-S\right)\overline{x}
\end{align*}
If $\mathbf{f}\subset \Lambda$ and $\overline{\mathbf{f}}\subset \overline{\Lambda}$ are basis, then
\begin{equation}\label{eq:sym_basis}
\left\{\begin{matrix} \mathbf{f}=\mathbf{e}~M+\overline{\mathbf{e}}~S~M \\ \overline{\mathbf{f}}=\mathbf{e}~\overline{M}+\overline{\mathbf{e}}~\overline{S}~\overline{M} \end{matrix} \right.
\end{equation}
where the columns of $M$ and $\overline{M}$ correspond to the coordinates of the vectors $\mathbf{f}$ and $\overline{\mathbf{f}}$ respectively. Hence, $\left(\mathbf{f},\overline{\mathbf{f}}\right)$ is a symplectic basis of $\left(W,\omega\right)$ whenever:
\[
M^{T}(\overline{S}-S) \overline{M} = I_n \, .
\]
Since $\overline{\Lambda}\in \Lambda^{\pitchfork}$, then $\overline{S}-S$ is a non--singular matrix and because $M$ is also non--singular, then we have that 
\begin{equation}\label{eq-base-simplectica}
\overline{M} = (\overline{S}-S)^{-1}\left(M^{T}\right)^{-1}
\end{equation}
is the matrix defining the basis $\overline{\mathbf{f}}$ such that $\left(\mathbf{f},\overline{\mathbf{f}}\right)$ is a symplectic basis of $W$.
\end{proof}

Choosing a symplectic basis $(\mathbf{e}, \bar{\mathbf{e}})$, the elements in the Lie algebra $\mathfrak{sp}$ are represented by matrices $A$ such that $A^T J + J A = 0$, with $J$ the matrix representation of the symplectic structure itself, that is the matrix whose non-zero elements are $J_{i,j +n} = \delta_{ij} = - J_{i+n,j}$, $i,j= 1,  \ldots, n$.  Then, $A$ has the form:
$$
A = \left( \begin{array}{cc} M & N \\ R & -M^T \end{array} \right) \, , \qquad N^T = N \, , \quad R^T = R \, ,
$$
with $M,N,R$ being $n\times n$ matrices, that shows that $\dim \mathfrak{sp} = 2n^2 + n$, which is the dimension of the symplectic group $Sp$, \textit{cfr.} Sect.  \ref{sec:symplectic_group}.

There is a natural isomorphism $S \colon \mathfrak{sp} \to \mathcal{S}(W)$, where $\mathcal{S}(W)$ denotes the linear space of symmetric bilinear forms on $W$.  Given $A \in \mathfrak{sp}$, we define, recall Eq. (\ref{eq:lie_algebra}):
$$
S_A(u,v) = ( Au | v) = (Av| u) \, .
$$
Thus, given a Lagrangian subspace $\Lambda \in  \mathscr{L}(W)$, $\mathfrak{sp}(\Lambda)$ will denote the subspace of all elements $A$ in $\mathfrak{sp}$ whose range lies in $\Lambda$, that is $\mathfrak{sp}(\Lambda) = \{ A \in \mathfrak{sp} \mid  A(W) \subset \Lambda \}$, or, equivalently:
\begin{equation}\label{eq:spLambda}
\mathfrak{sp}(\Lambda) = \{ A \in \mathfrak{sp} \mid  S_A ( \Lambda, \Lambda) = 0 \} \, .
\end{equation}
In other words, $\mathfrak{sp}(\Lambda) $ is isomorphic to the space of symmetric bilinear forms on $W$ vanishing on $\Lambda$, that is, $\mathcal{S}(W/\Lambda)$. But, $W/\Lambda$ is canonically isomorphic to $\Lambda^*$, the dual space to $\Lambda$, thus we conclude that $\mathfrak{sp}(\Lambda)$ is canonically isomorphic to $\mathcal{S}(\Lambda^*)$.


\subsection{The local affine structure of the Lagrangian Grassmannian}\label{sec:affine}

\subsubsection{The affine structure of $\Lambda^\pitchfork$}
Given a Lagrangian subspace $\Lambda \in  \mathscr{L}(W)$, the set $\Lambda^\pitchfork$, \textit{cfr.} (\ref{eq:transverse}),  carries a canonical affine structure. The following paragraphs will be devoted to describe it and to introduce some useful notations.

Given $\Gamma \in \Lambda^\pitchfork$ consider the canonical projector onto $\Lambda$ determined by the decomposition $W = \Lambda \oplus \Gamma$.  Using the convenient notation introduced in \cite{Ag02}, we will denote such projector, called a Lagrangian projector, by:
$$
\langle W, \Gamma, \Lambda \rangle \colon W \to \Lambda \subset W\, , \qquad \langle W, \Gamma, \Lambda \rangle(w) = w_\Lambda
$$   
where $w = w_\Lambda + w_\Gamma$, is the unique decomposition of the vector $w \in W$ on its $\Lambda$ and $\Gamma$ components, $w_\Lambda \in \Lambda$, $w_\Gamma \in \Gamma$.  Note that if $\Gamma \in \Lambda^\pitchfork$, then $\Lambda \in \Gamma^\pitchfork$, and we can define also the Lagrangian projector $\langle W,\Lambda, \Gamma \rangle$ with range $\Gamma$ along $\Lambda$, then:
$$
w = w_\Lambda + w_\Gamma = \langle W, \Gamma, \Lambda \rangle (w) + \langle W, \Lambda, \Gamma \rangle (w) \, , \quad \forall w \in W\,.
$$
There is a one-to-one correspondence between $\Lambda^\pitchfork$ and the set of Lagrangian projectors $\mathcal{P}_\Lambda$ associated to the Lagrangian subspace $\Lambda$, that is the set\footnote{Note that if $P$ is a Lagrangian projector, then $(Pu | v) = (u_\Lambda | v_\Gamma)$, and we get, $(Pu|v) + (u|Pv) = (u_\Lambda | v_\Gamma )  + (u_\Gamma | v_\Lambda )= (u|v)$, that implies that $P$ belongs to the Lie algebra of the conformal symplectic group, \textit{cfr.} (\ref{eq:conf_Lie_algebra}).}:
$$
\mathcal{P}_\Lambda = \{ P \in \mathfrak{gl}(W) \mid P^2 = P, P(W) = \Lambda, (Pu|v) + (u|Pv) = (u |v), \,  \forall u,v \in W \} \, ,
$$
The correspondence is such that, to any Lagrangian subspace $\Gamma \in \Lambda^\pitchfork$, we associate the Lagrangian projector $\langle W, \Gamma, \Lambda \rangle$ and, to any Lagrangian projector $P \in \mathcal{P}_\Lambda$, we associate the Lagrangian subspace $\Gamma = (I-P)(W)$.

\begin{lemma} 
The set  $\mathcal{P}_\Lambda$ is an affine space over the linear space $\mathfrak{sp}(\Lambda)$, \textit{cfr.} Eq. (\ref{eq:spLambda}). 
\end{lemma}
 \begin{proof}
  Indeed, given $P \in \mathcal{P}_\Lambda$ and $A\in \mathfrak{sp}(\Lambda)$, we observe that $Q = P + A$ is a Lagrangian projector.    Certainly, $Q(u) = P(u) + A(u)  \in \Lambda$.  On the other hand $Q^2 = (P+A)^2 = P+A = Q$, because $A^2 (u) = 0$ (note that $(Au|v) = (Av|u)$, thus if $Au, Av \in \Lambda$, then $(A^2u|v) = 0$).  Moreover $AP(u) = 0$, and $PA(u) = A(u)$.   Finally, $(Qu|v) + (u| Qv) = (Pu|v) + (u|Pv ) = (u|v)$, then $Q \in \mathcal{P}_\Lambda$.  
  \end{proof}

The natural identification of $\Lambda^\pitchfork$ with the space of Lagrangian projectors $\mathcal{P}_\Lambda$, introduces an affine structure on $\Lambda^\pitchfork$ with underlying linear space $\mathfrak{sp}(\Lambda)$, in other words, given two Lagrangian spaces $\overline{\Lambda}$ and $\Gamma$ in $\Lambda^\pitchfork$, the corresponding vector $\overrightarrow{\overline{\Lambda}\Gamma}$ will be given by the element in $\mathfrak{sp}(\Lambda)$:  $\overrightarrow{\overline{\Lambda}\Gamma} =  \langle W, \Gamma, \Lambda \rangle -  \langle W, \bar{\Lambda}, \Lambda \rangle$.  By definition, the value of any element belonging to $\mathfrak{sp}(\Lambda)$ in $\Lambda$ is zero, \textit{cfr.} (\ref{eq:spLambda}), then, we can introduce the convenient notation $\langle \bar{\Lambda}, \Gamma, \Lambda \rangle$ for the vector $\overrightarrow{\overline{\Lambda}\Gamma}$ that takes into account this fact.   Now $\langle \bar{\Lambda}, \Gamma, \Lambda \rangle$ is a linear map $\langle \bar{\Lambda}, \Gamma, \Lambda \rangle \colon \bar{\Lambda} \to \Lambda$, such that $(Au \mid v ) + (u \mid Av) = 0$ (compare with Eq. (\ref{eq:Lambda})).   In other words, given the Lagrangian subspace $\bar{\Lambda}\in \Lambda^\pitchfork$, using it as the origin in the affine space $\Lambda^\pitchfork$, we can identify the linear space $\mathfrak{sp}(\Lambda)$ with the space (see Fig. \ref{fig:projectors}):
\begin{equation}\label{eq:vector_space}
\mathcal{S}(\bar{\Lambda}, \Lambda) = \{ A \colon \bar{\Lambda} \to \Lambda \mid (A(u) \mid v ) + (u \mid A(v)) =0\} \, .
\end{equation}

\begin{figure}
  \centering
    \includegraphics[scale=0.5]{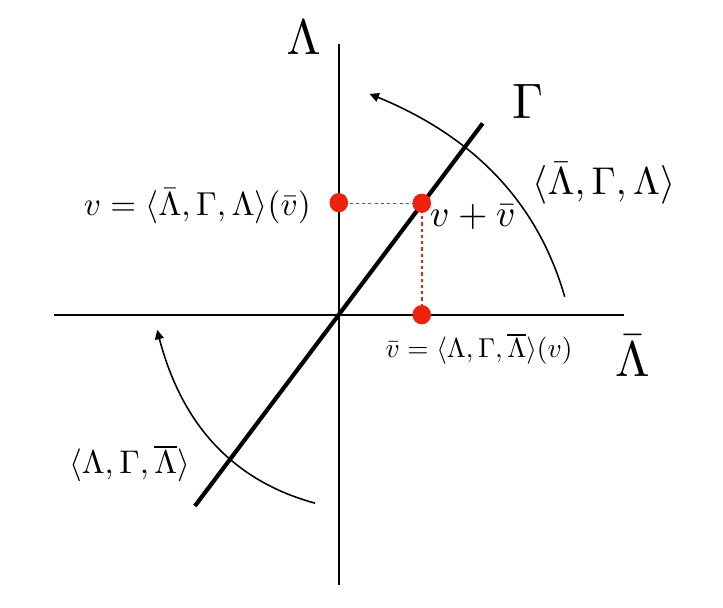}
  \caption{\small{A diagram representing the vectors $\langle \bar{\Lambda}, \Gamma, \Lambda \rangle$ (in the affine space $\Lambda^\pitchfork$ with origin at $\bar{\Lambda}$), and $\langle \Lambda, \Gamma, \bar{\Lambda} \rangle$ (in the affine space $\overline{\Lambda}^\pitchfork$ with origin at $\Lambda$).} \label{fig:projectors}}
\end{figure} 

Exchanging the roles of $\Lambda$ and $\bar{\Lambda}$, if $\Gamma_1, \Gamma_2$ are two Lagrangian subspaces in $\bar{\Lambda}^\pitchfork$, then the vector defined by them with respect to the affine structure on $\bar{\Lambda}^\pitchfork$, will be given by:
$$
\overrightarrow{\Gamma_1\Gamma_2} = \langle \Lambda, \Gamma_2 , \bar{\Lambda} \rangle - \langle \Lambda, \Gamma_1 , \bar{\Lambda} \rangle \, ,
$$
where now, $\Lambda$ plays the role of origin, the associated linear space is $\mathfrak{sp}(\bar{\Lambda})$ which is identified with $\mathcal{S}(\Lambda, \bar{\Lambda})$. 

\subsubsection{Coordinate description of the local affine structure.}\label{sec:coordinates}
It is convenient to emphasise the coordinate description of the previous constructions in terms of the atlas described in \S \ref{sec:preliminaries}, Eq. (\ref{eq:eq_Lag}).   We will introduce as before a symplectic basis $(\mathbf{e}, \bar{\mathbf{e}})$ with corresponding symplectic coordinates $(x,\bar{x})$, and the Lagrangian subspaces $\Lambda = \langle e_1, \ldots, e_n \rangle$, and $\bar{\Lambda} = \langle \bar{e}_1, \ldots, \bar{e}_n \rangle$.   Thinking of the vector $x \in \mathbb{R}^n$ as a column vector, the symplectic form $\omega$ takes the coordinate expression:
\begin{equation}\label{eq:coor_sym}
((x_1, \bar{x}_1) \mid (x_2, \bar{x}_2) ) = x_1^T \bar{x}_2 - x_2 \bar{x}_1^T \, .
\end{equation}
In particular, $((x, 0) \mid (0, \bar{x}) ) = x^T \bar{x}$.   Note that in the identification $\Lambda^* \cong \bar{\Lambda}$, the basis $\bar{\mathbf{e}}$ gets identified with the opposite dual basis of $\mathbf{e}$, that is:  $(\bar{e}_j | \,\cdot\, ) = - x^j$ and $(e_j | \,\cdot\, ) = \bar{x}^j$.  Moreover the space of symmetric bilinear forms $\mathcal{S}(\bar{\Lambda})$ is identified with the space of symmetric matrices  $\mathcal{S}(\mathbb{R}^n) = \{ \bar{S} = (\bar{S}_{ij}) \mid \bar{S}^T = \bar{S}\}$.

Then, the affine space $\Lambda^\pitchfork$ is identified (using as origin $\bar{\Lambda}$) with the vector space $\mathcal{S}(\mathbb{R}^n)$ by means of:
$$
\bar{S} = (\bar{S}_{ij}) \in \mathcal{S}(\mathbb{R}^n) \mapsto \Gamma_{\bar{S}} = \{ (\bar{S}\bar{x}, \bar{x} ) \mid \bar{x} \in \mathbb{R}^n \} \in \Lambda^\pitchfork \, ,
$$
and the map $S_{\Lambda^\pitchfork}^{\bar{\Lambda}} \colon \Gamma \in \Lambda^\pitchfork  \mapsto \langle \bar{\Lambda}, \Gamma, \Lambda \rangle \in \mathcal{S}(\bar{\Lambda}) \cong \mathcal{S}(\mathbb{R}^n)$, constitutes an affine coordinate system for $\Lambda^\pitchfork$ (see Fig. \ref{fig:projectors}).    Similarly, the affine space $\bar{\Lambda}^\pitchfork$, can be identified (using as origin $\Lambda$) with the vector space $\mathcal{S}(\mathbb{R}^n)$ again, by means of (compare with (\ref{eq:eq_Lag})):
\begin{equation}\label{eq:S_coordinate}
S = (S_{ij}) \in \mathcal{S}(\mathbb{R}^n) \mapsto \Gamma_S = \{ (x, Sx) \mid x \in \mathbb{R}^n \} \in \bar{\Lambda}^\pitchfork \, ,
\end{equation}
where $S$ is interpreted as $\langle \Lambda, \Gamma_S, \bar{\Lambda} \rangle$.  The map $S_{\bar{\Lambda}^\pitchfork}^\Lambda \colon \Gamma \in \bar{\Lambda}^\pitchfork  \mapsto \langle \Lambda, \Gamma, \bar{\Lambda} \rangle \in \mathcal{S}(\Lambda) \cong \mathcal{S}(\mathbb{R}^n)$, $S_{\bar{\Lambda}^\pitchfork}^\Lambda \cong (S_{ij} \mid i \leq j)$, constitutes an affine coordinate system for $\bar{\Lambda}^\pitchfork$.


\section{Jacobi curves and their classification}\label{sec:Jacobi}

\subsection{Jacobi curves}\label{sec:jacobi_curves}

A curve in $ \mathscr{L}(W)$ is a smooth map $\Gamma$ from an open interval $I \subset \mathbb{R}$ in  $ \mathscr{L}(W)$.  We will denote the curve $\Gamma$ as $\Gamma = \Gamma (t)$.   The tangent vector to the curve $\Gamma$ at $\Gamma(t)$ will be denoted as $\Gamma'(t) \in T_{\Gamma(t)}  \mathscr{L}(W)$.

In coordinates $S = (S_{ij}) \in \mathcal{S}(\mathbb{R}^n)$ introduced above, \textit{cfr}. \ref{sec:coordinates}, Eq. (\ref{eq:S_coordinate}), we can write a Jacobi curve as:
$$
\Gamma (t) = \left[ \begin{array}{c} I_n \\ S_t \end{array} \right]_{(\mathbf{e}, \bar{\mathbf{e}})} = \left\{ (\mathbf{e}, \bar{\mathbf{e}}) \cdot \left(\begin{array}{c} x \\ S_tx \end{array} \right) \mid x \in \mathbb{R}^n \right\} \, .
$$
where as indicated before, \textit{cfr.} (\ref{eq:eq_Lag}), the notation  $\left[ \begin{array}{c} I_n \\ S_t \end{array} \right]_{(\mathbf{e}, \bar{\mathbf{e}})} $ describes the Lagrangian subspace spanned by the vectors $e_i + \sum_{j= 1}^n(S_t)_{ji}\bar{e}_j$, $i = 1,\ldots, n$.

Let us consider a Lagrangian subspace $\Lambda_0\in \mathscr{L}\left(W\right)$ and denote as before by $\mathcal{S}\left(\Lambda_0\right)$ the vector space of symmetric bilinear forms defined in $\Lambda_0$.    There is a natural characterization of the tangent space $T_{\Lambda_0}( \mathscr{L}(W))$ as the space of quadratic forms on $\Lambda_0$.

\begin{proposition}\label{prop-bilineal-tangente}
Let $\Gamma=\Gamma(t)$ be a curve in $\mathscr{L}\left(W\right)$ such that $\Gamma(0)=\Lambda_0$. We define the quadratic form $\dot{\Gamma}(0):\Lambda_0 \rightarrow \mathbb{R}$ by
\[
\dot{\Gamma}(0) \left(v_0\right) = \omega\left(v_0,v'(0)\right)
\]
where $v=v(t)\in\Gamma(t)$ is a curve such that $v(0)=v_0$ and $\omega$ the symplectic 2--form in $W$. Then the map 
\[
\begin{tabular}{rcl}
$T_{\Lambda_0}\mathscr{L}(W)$ & $\rightarrow$ & $\mathcal{S}\left(\Lambda_0\right)$ \\
$\Gamma'(0)$ & $\mapsto$ & $\dot{\Gamma}(0)$
\end{tabular}
\]
is a linear isomorphism.
\end{proposition}

\begin{proof}
First, we will show that the quadratic form $\dot{\Gamma}(0)$ does not depend on the chosen curve $v=v(t)$. If $u=u(t)\in\Gamma(t)$ is another curve such that $u(0)=v_0$, then $v(t)-u(t)\in \Gamma(t)$ for all $t$. Since $\Gamma(t)\in \mathscr{L}(W)$ is Lagrangian, then  $\left(v(t) \mid v(t)-u(t)\right)=0$.   Deriving at $t=0$, we obtain:  $\left(v'(0) \mid v(0)-u(0)\right) + \left(v(0) \mid v'(0)-u'(0)\right)=0
\Rightarrow~  \left(v(0) \mid v'(0)-u'(0)\right)=0$, but then
$\dot{\Gamma}(0) \left(v_0\right) = \left(v(0) \mid u'(0)\right)$,
that shows that $\dot{\Gamma}(0)$ does not depend on the chosen $v$.

Now, consider a coordinate system $S = (S_{ij})$ of $\mathscr{L}(W)$ at $\Lambda_0$ such that $\Gamma(t)=\left\{ \left(x,S_t x\right) \mid  x\in \mathbb{R}^{m} \right\}$,  where, $S_t =\left(S_{ij}(t)\right)\in \mathbb{R}^{n \times n}$.
Since $\Lambda_0 = \left\{ \left(x,S_0 x\right): x\in \mathbb{R}^{m} \right\}$, then there exists a unique $x_0\in\mathbb{R}^m$ such that $v_0 = \left(x_0,S_0 x_0\right)$. We can consider $v$ as the curve defined by: $v(t)=\left(x_0,S_t x_0\right)$,
so $v'(t)= \left(0,S'_t x_0\right)$, and then:
\[
\dot{\Gamma}(0) \left(v_0\right) = \left(v_0 \mid v'(0)\right) = \left( \left(x_0,S_t x_0\right) \mid \left(0,S'_t x_0\right) \right) =  x_0^T S'_0x_0
\]
which only depends on the tangent vector $\Gamma'(0)\in T_{\Lambda_0}\mathscr{L}(W)$, that in local coordinates has the expression:
\[
\Gamma'(0)=\sum_{i\leq j}S'_{ij}(0)\left( \frac{\partial}{\partial S_{ij}} \right)_{\Lambda_0}
\] 
and the coordinates $S'_{ij}(0)$ coincide with the components of the matrix $S'_0$ defining $\dot{\Gamma}(0)$. Therefore the map $T_{\Lambda_0}\mathscr{L}(W)\rightarrow \mathcal{S}\left(\Lambda_0\right)$ is a linear isomorphism.
\end{proof}

\begin{remark}\label{rem:Sprime}
The previous proposition is also true for any value of $t$.   Note that $S'_\tau = (S'_{ij}(\tau))$ is the matrix of $\dot{\Gamma}(\tau)$ in the basis of $\Gamma(\tau)$:  $\boldsymbol{\epsilon}(\tau) = \mathbf{e} + \overline{\mathbf{e}} S_\tau$, that provides the system of coordinates $x \in \mathbb{R}^n \mapsto (x, S_\tau x) \in \Gamma(\tau)$.  As an immediate consequence we get that $\dot{\Gamma}(\tau)$ is positive definite iff the matrix $S'(\tau)$ is positive definite.
\end{remark}

Given a curve $\Gamma (t)$ in the Lagrangian Grassmannian, in what follows we will identify $\Gamma'(t)$ with the quadratic form $\dot{\Gamma}(t)$.   We will assume also the following transversality conditions in increasing order of restrictiveness:
\begin{enumerate}
\item $\Gamma (t)$ is non-singular, that is $\dot{\Gamma} (t) \neq 0$, for all $t$.
\item $\Gamma (t)$ is regular if $\dot{\Gamma}(t)$ is a non-degenerate quadratic form on each Lagrangian subspace $\Gamma(t)$.
\item $\Gamma (t)$ is monotonous if $\dot{\Gamma}(t)$ is definite (positive or negative).
\end{enumerate}

In what follows we will just consider regular curves in the Lagrangian Grassmannian and, unless there is risk of confusion, we will call them Jacobi curves.

\begin{definition}\label{def:jacobi}
A regular Jacobi curve $\Gamma$ (or just a Jacobi curve for short) on the (conformal) symplectic space $W$ is a smooth regular curve in the Lagrangian Grassmannian $\mathscr{L}(W)$.
\end{definition}


\subsection{The classification of Jacobi curves: curvatures}

As it was discussed in Sect. \ref{sec:symplectic_group} there is a natural transitive action of the conformal symplectic group $CSp$ in the Lagrangian Grassmannian $(a,\Lambda) \mapsto a\Lambda$, for all $a \in CSp$ and $\Lambda \in  \mathscr{L}(W)$.   The same action induces an action on the set of Jacobi curves, that is if $\Gamma = \Gamma (t)$ is a Jacobi curve then $\tilde{\Gamma} = a\Gamma = a (\Gamma (t))$ is also a Jacobi curve.  Notice that the action of $CSp$ on the Lagrangian Grassmaniann is the same as the action of the symplectic group $Sp$.   More generally will say that the Jacobi curves $\tilde{\Gamma}$ and $\Gamma$ are $CSp$-equivalents if there is a change of parameter $t = \varphi (\tilde{t})$ and an element $a \in CSp$, such that $a \Gamma (t) = a \Gamma (\varphi (\tilde{t})) = \tilde{\Gamma}(\tilde{t})$.   This implies that the map $[\Gamma] \to [\tilde{\Gamma}]$ between the graphs of the curves $\Gamma$ and $\tilde{\Gamma}$, given by $\Gamma(t) \mapsto \tilde{\Gamma}(\varphi^{-1}(t))$ can be obtained by restriction of the conformal symplectic map $a$. 

The problem of classification of Jacobi curves consists of finding an algorithmic criterion that allows to decide when two Jacobi curves are $CSp$-equivalent (or just `equivalent' if there is no risk of confusion).  

\paragraph{Parametric curvatures.} Curvatures could provide the criterion we are looking for.   A parametric curvature is a function $\kappa$ that assigns to any Jacobi curve $\Gamma = \Gamma (t)$, a differentiable function $\kappa_\Gamma = \kappa_\Gamma (t)$ such that $\kappa_{a\Gamma} = \kappa_\Gamma$, for all $a \in CSp$ and $\kappa_\Gamma (t) = \kappa_\Gamma (t + t_0)$, for all $t,t_0$ such that $t+t_0$ belongs to the domain of $\Gamma$.  We will say that the parametric curvature $\kappa$ is of order $\leq r$, if $\kappa_\Gamma$ is a differentiable function of the derivatives of order $\leq r$ of $\Gamma$.   We will say that $\kappa$ is of order $r$ if it is of order $\leq r$ and it is not of order $\leq r -1$.

\paragraph{Absolute curvatures} Given a parametric curvature $\kappa$, then for a given Jacobi curve $\Gamma$, the function $\kappa_\Gamma (t)$ will change in a specific manner under changes of parameter $t = \varphi (\tilde{t})$.   In general $\kappa_{\tilde{\Gamma}}(\tilde{t})$ will not coincide with $\kappa_\Gamma (\varphi(\tilde{t}))$.   When the parametric curvature $\kappa$ is consistent with respect to changes of parameters we will say that $\kappa$ is an absolute curvature.

\begin{definition}
Let $\kappa$ be a parametric curvature for the class of Jacobi curves.  We will say that $\kappa$ is an absolute curvature if $\kappa_{\tilde{\Gamma}}(\tilde{t}) = \kappa_\Gamma (\varphi(\tilde{t}))$ for all changes of parameters $t = \varphi (\tilde{t})$.
\end{definition}

If $\kappa$ is an absolute curvature, then for each Jacobi curve $\Gamma$ it induces a real function $\kappa_{[\Gamma]} \colon [\Gamma] \to \mathbb{R}$, on the graph of $\Gamma$, given by $\kappa_{[\Gamma]} (\Lambda) = \kappa_\Gamma (t)$, where $\Lambda = \Gamma (t)$.   Finally the invariance of parametric curvatures with respect to the action of the conformal symplectic group implies that the map $\kappa_{[\Gamma]}$ is $CSp$-invariant, that is $\kappa_{[\Gamma]} (\Lambda) =  \kappa_{[a\Gamma]} (a\Lambda)$, $a \in CSp$.   The absolute curvature $\kappa$ will be called just a curvature (unless there is risk of confusion).

\paragraph{Complete system of curvatures}  A family $\{ \kappa^1, \ldots, \kappa^r\}$ of absolute curvatures will be called a complete system of curvatures for the class of Jacobi curves if for any pair of Jacobi curves $\Gamma,\tilde{\Gamma} \colon I \to \mathscr{L}(W)$, such that $\kappa_\Gamma^i = \kappa_{\tilde{\Gamma}}^i$, for all $i = 1, \ldots , r$, then $\Gamma$ and $\tilde{\Gamma}$ are equivalent.   Note that this implies that any diffeomorphism $\alpha \colon [\Gamma] \to [{\tilde{\Gamma}}]$ that preserves the curvatures $\kappa^i$ is the restriction of a conformal symplectic transformation $a \in CSp$.

We will say that a system of curvatures  $\{ \kappa^1, \ldots, \kappa^r\}$ is independent if the functions $\kappa^i$ are functionally independent.   Then, a complete family of independent curvatures will provide a classification of Jacobi curves.    We will devote the remaining of this article to construct one such family of curvatures, thus solving the problem of classification of Jacobi curves.   It is relevant to point out that the construction to be described is algorithmic.


\subsection{Geometric arc parameters}\label{sec:geometric_arc}

\begin{definition}\label{def:geometric_arc}
A geometric arc parameter is a map $ds_\bullet$ that assigns to each Jacobi curve $\Gamma$ an arc element $ds_\Gamma = \zeta_\Gamma (t) dt$ (that is, a 1-form along the curve $\Gamma = \Gamma (t)$), invariant under the action of the conformal symplectic group $CSp$, i.e., $ds_\Gamma = ds_{a\Gamma}$, for all $a \in CSp$, and if $t = \varphi(\bar{t})$ is a change of parameter, and $\bar{\Gamma} = \Gamma (\varphi (\bar{t}))$ is the reparametrised Jacobi curve, then $\varphi^* ds_\Gamma = ds_{\overline{\Gamma}}$, namely:
$$
\zeta_{\bar{\Gamma}}(\bar{t}) = \zeta_\Gamma (\varphi (\bar{t})) \frac{d\varphi}{d\bar{t}} \, .
$$
\end{definition}

Note that if $ds_\bullet$ is a geometric arc parameter, then the length $L(\Gamma)$ of the curve does not depend on the parameter we use to parametrise it, namely:
$$
L(\Gamma) = \int_a^b \zeta_\Gamma (t) dt = \int_{\bar{a}}^{\bar{b}} \zeta_{\bar{\Gamma}}(\bar{t}) d\bar{t} = L(\bar{\Gamma}) \, .
$$
and, in addition, $L (\Gamma) = L (a \Gamma) $, for all $a \in CSp$.


\section{The Ricci curvature tensor of a Jacobi curve: the derivative curve}\label{sec:Ricci}

\subsection{The derivative curve and its symmetry} 

The construction and definition of the Ricci curvature tensor will rely on the notion of the derivative curve of a Jacobi curve.  We will present here a construction different from the one that appears in \cite{Ag02} that is, we believe, perhaps simpler and more natural.  We will illustrate it first in the bidimensional case, where the Lagrangian Grassmannian is just the projective line.


\subsubsection{The derivative curve in the bidimensional case}

If $\dim W = 2$, then for every $w \in W \backslash \{ \mathbf{0}\}$, we get $(w | w) = 0$, hence the vector line $[w] \in \mathbb{P}(W)$ defines a Lagrangian subspace and $ \mathscr{L}(W) = \mathbb{P}(W)$.   A symplectic basis is just a pair of vectors $e,\bar{e}$ such that $(e \mid \bar{e} ) = 1$.  Symplectic coordinates $(x,\bar{x})$ with respect to such basis allow us to identify $ \mathscr{L}(W)$ with $\mathbb{RP}^1$ by considering the Lagrangian subspaces $\Lambda = [e] = (1:0)$, and $\bar{\Lambda} = [\bar{e}] = (0:1)$.  Under these circumstances the affine line structure for $\bar{\Lambda}^\pitchfork \cong \mathbb{RP}^1 \backslash \bar{\Lambda}$, is given by the Cartesian coordinate $S \in \mathbb{R} \mapsto (1:S) \in \bar{\Lambda}^\pitchfork$, with origin in $\Lambda$.   Thus, a Jacobi curve $\Gamma(t)$ in $\mathbb{RP}^1 \backslash \bar{\Lambda}$ is given in projective coordinates by $(1:S_t)$, with $S_t' \neq 0$ for all $t$.   Thus, we get:

\begin{proposition}\label{prop:derivative2}
Given a fixed value of the parameter $\tau$, there exists a unique point $\Delta_\tau \in  \mathscr{L}(W)$, $\Delta_\tau \neq \Gamma (t)$, for any $t$,  such that the curve $\Gamma (t)$ in the affine space $\Delta_\tau^\pitchfork$, denoted as $\Gamma_{\Delta_\tau}$, satisfies $\Gamma''_{\Delta_\tau} (\tau)= 0$.  More specifically, if $S_\tau'' \neq 0$, then:
\begin{equation}\label{eq:Atau1}
\Delta_\tau = \left(S_\tau'': S_\tau S_\tau'' - 2(S_\tau')^2\right) \, ,
\end{equation}
and, obviously, if $S_\tau'' = 0$, we get $\Delta_\tau = \bar{\Lambda} = (0:1)$.
\end{proposition}

The assignment $\tau \mapsto \Delta_\tau$ will be called the derivative curve of the Jacobi curve $\Gamma (t)$.  The proof that such curve exists in this bidimensional setting is simple enough.

\begin{proof} Given $A = (1:a) \in \bar{\Lambda}^\pitchfork = \mathbb{RP}^1 \backslash \{ \bar{\Lambda} \}$, the Cartesian coordinate $\tilde{S}$ in $A^\pitchfork = \mathbb{RP}^1 \backslash \{ A\}$ and origin in $S_\tau$ is given as a function of the Cartesian coordinate $S$  in 
$\bar{\Lambda}^\pitchfork$, as:
$$
\tilde{S} = \frac{S - S_\tau}{S - a} \, .
$$
Then, substituting the Cartesian coordinate $S$ by $S_t$, we get the equation for the line $\Gamma(t)$ in the affine space $A^\pitchfork$:
$$
\tilde{S}_t = \frac{S_t - S_\tau}{S_t - a} \, .
$$
We are looking for a number $a$ such that $\tilde{S}''_\tau = 0$.    Writing $f(t) = S_t - S_\tau$, we get:
$$
\tilde{S}_t = \frac{f(t)}{f(t) + S_\tau - a} =  \frac{f(t)}{f(t) - b} \, , 
$$
with $b = a - S_\tau$.  Then,
$$
\tilde{S}_\tau' = - \frac{bf'}{(f-b)^2} \, , \quad \mathrm{and} \qquad \tilde{S}_\tau'' = -b \frac{f''(f-b) - 2(f')^2}{(f-b)^3} \, .
$$
Thus, if we wish $\tilde{S}''_\tau = 0$, we get: $f''(\tau) (f(\tau) - b) - 2 f'(\tau)^2 = 0$.    Substituting $f(\tau ) = 0$, $f'(\tau) = S'_\tau$, $f''(\tau) = S_\tau''$, and $b = a - S_\tau$, we obtain the desired expression:
$$
a = S_\tau - \frac{2 (S_\tau')^2}{S_\tau''} \, ,
$$
that corresponds to the point with projective coordinates (\ref{eq:Atau1}).
\end{proof}


\subsubsection{The definition and construction of the derivative curve in the general case}

Using as a guide the construction described in the previous section, we will proceed to construct the analogue of the curve $\tau \mapsto \Delta_\tau$, \textit{cfr.} (\ref{eq:Atau1}), when $\dim W = 2n$.  As discussed in \S \ref{sec:coordinates}, \textit{cfr.} (\ref{eq:S_coordinate}), let us consider a coordinate system $S=S^{\Lambda}_{\overline{\Lambda}^{\pitchfork}}$ on the affine space $\overline{\Lambda}^{\pitchfork}\in \mathscr{L}(W)$, with origin at $\Lambda$, that is $S(\Lambda)=0$.
For a given regular Jacobi curve $\Gamma=\Gamma(t)$ with $t\in I \subset \mathbb{R}$, we have that $S_t :=S\left(\Gamma(t)\right)$ is differentiable and the matrix $S'_t\in\mathbb{R}^{n \times n}$ is regular in any coordinate system $S^{\Lambda}_{\overline{\Lambda}^{\pitchfork}}$ with $\Lambda\in\overline{\Lambda}^{\pitchfork}$. 

\begin{remark}\label{rmk-coordinates-S}
Note that choosing the parameter $\tau\in I$ such that $\Gamma(\tau)\in\Lambda^{\pitchfork} \cap \overline{\Lambda}^{\pitchfork}$, then $S-S_{\tau}=S^{\Gamma(\tau)}_{\overline{\Lambda}^{\pitchfork}}$, with $S_{\tau}$ the matrix of coordinates of $\Gamma(\tau)$,  is a coordinate system on $\overline{\Lambda}^{\pitchfork}$ with origin at $\Gamma(\tau)$. 
Hence $(S-S_{\tau})^{-1}=S^{\overline{\Lambda}}_{\Gamma(\tau)^{\pitchfork}}$ are coordinates on $\Gamma(\tau)^{\pitchfork}\cap \Lambda^{\pitchfork} \cap \overline{\Lambda}^{\pitchfork}$, and the matrix of coordinates $S-S_{\tau}$ is invertible at any Lagrangian subspace in $\Gamma(\tau)^{\pitchfork}\cap \Lambda^{\pitchfork} \cap \overline{\Lambda}^{\pitchfork}$.   Also, we can observe that the matrix $(S_t-S_{\tau})^{-1}$ are the coordinates of the curve $\Gamma=\Gamma(t)$ on $\Gamma(\tau)^{\pitchfork}\cap \Lambda^{\pitchfork} \cap \overline{\Lambda}^{\pitchfork}$ for $t$,  $\tau\neq t$, in an open interval that we can consider as the interval $I$. 
\end{remark}

Given a Jacobi curve $\Gamma$ and $\Delta\in \mathscr{L}(W)$ we will denote by $\Gamma_{\Delta}$ the restriction of the curve $\Gamma$ to the affine space $\Delta^{\pitchfork}\subset \mathscr{L}(W)$, i.e., $\Gamma_\Delta (t) = \Gamma (t)$, provided that $\Gamma(t) \in \Delta^{\pitchfork}$. Then the derivatives $\Gamma'_{\Delta}$, $\Gamma''_{\Delta}$, $\Gamma'''_{\Delta}$, ... of $\Gamma_{\Delta}$ are defined in the vector space $\overrightarrow{\Delta}^{\pitchfork}$ associated to $\Delta^{\pitchfork}$.   

Consider now a Jacobi curve $\Gamma (t)$, then fixed a point $\Gamma (\tau)$ of the curve, and $\Delta \in \Gamma (\tau)^\pitchfork$, there exists $\epsilon > 0$, such that $\Gamma (t) \in \Delta^\pitchfork$ for all $|t - \tau | < \epsilon$, and we denote as indicated before, by $\Gamma_\Delta$ the restriction of $\Gamma$ to $\Delta^\pitchfork$.
Under these conditions we get the following theorem that extends our previous result, Prop. \ref{prop:derivative2}:

\begin{theorem}\label{thm-curva-derivativa}
There exists a unique $\Delta_\tau\in \Gamma(\tau)^{\pitchfork}$ such that $\Gamma''_{\Delta_\tau} (\tau) = 0$.
\end{theorem}

The proof of Thm.  \ref{thm-curva-derivativa} has two parts.   First, assuming its existence, an explicit expression for the coordinates of $\Delta_\tau$ will be given.

\begin{lemma}\label{lem-derivative-1}
For fixed $t=\tau$, let us assume that there exists $\Delta_\tau\in \Gamma(\tau)^{\pitchfork}$ such that $\Gamma''_{\Delta_\tau}=0$. If $S^0_{\tau}=S^{\Lambda}_{\overline{\Lambda}^{\pitchfork}}\left( \Delta_\tau \right)$, then both $S^0_{\tau}-S_{\tau}$ and $S''_{\tau}$, are invertible and we get:
\begin{equation}\label{eq:Atau}
S^0_{\tau} = S_{\tau} - 2 S'_{\tau} (S''_{\tau})^{-1} S'_{\tau}
\end{equation}

\end{lemma}

\begin{proof}
Recall that we can write 
\[
S^{\Lambda}_{\overline{\Lambda}^{\pitchfork}}\left(\Gamma\right)=\left\langle \Lambda,\Gamma,\overline{\Lambda} \right\rangle : \Lambda \longrightarrow\overline{\Lambda}  .
\]

Moreover we have $\left\langle \overline{\Lambda},\Gamma,\Lambda \right\rangle = \left\langle \Lambda,\Gamma,\overline{\Lambda} \right\rangle^{-1}$, and then, we have:
\[
\left\{
\begin{tabular}{l}
$S_t - S_{\tau} = S^{\Gamma(\tau)}_{\overline{\Lambda}^{\pitchfork}}\left(\Gamma(t)\right) $  \\
\\
$S^0_{\tau} - S_{\tau} = S^{\Gamma(\tau)}_{\overline{\Lambda}^{\pitchfork}}\left(\Delta_\tau\right) $
\end{tabular}
\right.
\]
then $(S^0_{\tau} - S_{\tau})^{-1} = S^{\overline{\Lambda}}_{\Gamma(\tau)^{\pitchfork}}\left(\Delta_\tau\right) $ are the coordinates of $\Delta_\tau$ in the affine space $\Gamma(\tau)^{\pitchfork}$. So, $S^0_{\tau} - S_{\tau}$ is invertible. 

Fixed $\Gamma(\tau)$, now we will seek the expression of $S^0_{\tau}$. 
Observe that:
\begin{equation}\label{eq:Stilde}
\widetilde{S}_t = \left(\left(S_t - S_{\tau}\right)^{-1} - \left(S^0_{\tau}-S_{\tau}\right)^{-1}\right)^{-1} = S^{\Gamma(\tau)}_{\Delta_\tau^{\pitchfork}}\left(\Gamma(t)\right):  \Gamma(\tau) \longrightarrow \Delta_\tau   
\end{equation}
corresponds to the coordinates of $\Gamma(t)$ in $\Delta_\tau^{\pitchfork}$ with origin at $\Gamma(\tau)$.
Let us compute the condition:
\[
\left.\widetilde{S}''_t\right|_{t=\tau}=0   .
\]

For short, in the following computation, we will call 
\[
\left\{
\begin{tabular}{l}
$B_t=S_t- S_{\tau}$ \\
$C = S^0_{\tau}-S_{\tau}$ \\
$D_t = S^0_{\tau} - S_t$
\end{tabular}
\right.
\]
and since $B'_t=S'_t=-D'_t$ and $C'=0$, substituting at $t=\tau$, we get: 
\begin{equation}\label{eq-B-C-D-tau}
\left\{
\begin{tabular}{lll}
$B_{\tau}=0$ & , & $B^{(i)}_{\tau}=S^{(i)}_{\tau}$ \\
$D_{\tau}=C$ & , & $D^{(i)}_{\tau}=-S^{(i)}_{\tau}$
\end{tabular}
\right.
\end{equation}
where $(i)$ denotes the $i$--th order derivative with $i=1,2,\ldots$

Notice that we can re--write (\ref{eq:Stilde}) as:
\begin{align}\label{eq-Stilde-1}
\widetilde{S}_t & = \left(\left(S_t - S_{\tau}\right)^{-1} - \left(S^0_{\tau}-S_{\tau}\right)^{-1}\right)^{-1} = \nonumber \\ 
& =  \left(B_t^{-1} - C^{-1}\right)^{-1} = \left(B_t^{-1} B_t\left(B_t^{-1} - C^{-1}\right)C C^{-1}\right)^{-1} = \nonumber \\ 
& =  \left(B_t^{-1} \left(C-B_t\right)C^{-1}\right)^{-1} = C \left(C-B_t\right)^{-1} B_t = C \left(D_t\right)^{-1} B_t =\\
& = \left( S^0_{\tau}-S_{\tau} \right)\left( S^0_{\tau}-S_t \right)^{-1} \left( S_t-S_{\tau} \right) \label{eq-Stilde-2}
\end{align}

From equation (\ref{eq-Stilde-1}), we have $\widetilde{S}_t = C \left(D_t\right)^{-1} B_t $, then:
\begin{align}
\widetilde{S}'_t &= C \left(-(D_t)^{-1}D'_t (D_t)^{-1}B_t + (D_t)^{-1}B'_t\right) = \nonumber \\
&= C (D_t)^{-1}S'_t \left((D_t)^{-1}B_t + \mathrm{Id}\right) \label{eq-S-prima} \, ,
\end{align}
and, substituting $t=\tau$ in equation (\ref{eq-B-C-D-tau}), we get
\begin{equation}\label{eq-Sprima-tau}
\widetilde{S}'_{\tau} = S'_{\tau} \, .
\end{equation}

Deriving again (\ref{eq-S-prima}), one obtains:
\begin{equation}\label{eq-S-prima-2}
\widetilde{S}''_t = C (D_t)^{-1}\left(2S'_t (D_t)^{-1}S'_t + S''_t\right)\left((D_t)^{-1}B_t + \mathrm{Id}\right) \, ,
\end{equation}
and, again for $t=\tau$:
\begin{equation}
\widetilde{S}''_{\tau}  = 2S'_{\tau} C^{-1}S'_{\tau} + S''_{\tau} = 
  2S'_{\tau} (S^0_{\tau}-S_{\tau})^{-1} S'_{\tau} + S''_{\tau}  \, .
\end{equation}

If $\widetilde{S}''_{\tau}=0$ then 
\begin{equation}\label{eq-Stau-A}
(S^0_{\tau}-S_{\tau})^{-1}= -\frac{1}{2}(S'_{\tau})^{-1}S''_{\tau}(S'_{\tau})^{-1}
\end{equation}
and since $S^0_{\tau}-S_{\tau}$ and $S'_{\tau}$ are invertible, so is $S''_{\tau}$.   Finally, we can re--write Eq. (\ref{eq-Stau-A}) as 
\begin{equation}\label{eq-A-tau}
S^0_{\tau}= S_{\tau}  -2S'_{\tau} (S''_{\tau})^{-1}S'_{\tau}    
\end{equation}
to obtain the expression of $S^0_{\tau}$ as claimed.
\end{proof}

Before the proof of the existence of $\Delta_\tau$ we will establish some technical facts. 

\begin{lemma}\label{lem-derivative-2}
Let $\varphi:I\rightarrow \mathbb{R}$ be a smooth function. Then the function 
\[
g(t,s)= \left\{
\begin{tabular}{ll}
$\frac{\varphi(t)-\varphi(s)}{t-s}$ & if $t\neq s$ \\
$\varphi'(t)$ & if $t = s$
\end{tabular}
\right.
\]
defined in a neighbourhood of the diagonal $D_I=\{(t,t):t\in I\}$ is smooth.  
\end{lemma}

\begin{proof}
By Barrow's rule, we have
\[
\varphi(t)-\varphi(s) = \int_0^1 \left.\frac{d}{d\lambda}\right|_{\lambda}\varphi(\lambda t + (1-\lambda) s) ~d\lambda = (t-s) \int_0^1 \varphi'(\lambda t + (1-\lambda) s) ~d\lambda 
\]
then 
\[
g(t,s)= \int_0^1 \varphi'(\lambda t + (1-\lambda) s) ~d\lambda
\] 
is the sought function.
\end{proof}

\begin{lemma}\label{lem-derivative-3}
Given a Jacobi curve $\Gamma=\Gamma(t) \simeq \begin{bmatrix} \mathrm{Id} \\ S_t \end{bmatrix}$, then the function  
\[
G(t,s)= \left\{
\begin{tabular}{ll}
$\frac{S_t-S_s}{t-s}$ & if $t\neq s$ \\
$S'_t$ & if $t = s$
\end{tabular}
\right.
\]
is smooth, symmetric and invertible with smooth inverse $F(t,s)=G(t,s)^{-1}$ in a neighbourhood of the diagonal $D_I$.  
\end{lemma}

\begin{proof}
If $S_t=\left( S_{ij}(t) \right)_{i,j=1}^{m}\in\mathbb{R}^{n \times n}$ is a smooth symmetric matrix, then we can apply Lemma \ref{lem-derivative-2} to the components $S_{ij}$ obtaining smooth functions 
\[
g_{ij}(t,s)= \left\{
\begin{tabular}{ll}
$\frac{S_{ij}(t)-S_{ij}(s)}{t-s}$ & if $t\neq s$ \\
$S'_{ij}(t)$ & if $t = s$
\end{tabular}
\right.
\]
So, the matrix $G(t,s)=\left( g_{ij}(t,s) \right)_{i,j=1}^{m}$ is smooth and the symmetry of $S_t$ implies the symmetry of $G$. 
Moreover, since 
\[
\mathrm{det}~G(t,t)=\mathrm{det}~S'_t\neq 0
\] 
then, by remark \ref{rmk-coordinates-S}, $G$ is invertible in a neighbourhood of $D_I$ and 
\[
G(t,s)^{-1}=(t-s)\left( S_t - S_s \right)^{-1}
\] 
which is also smooth in a neighbourhood of $D_I$.
\end{proof}

Recall that, because of transversality, the curve $\Gamma$ is also defined in the affine space $\Gamma(\tau)^{\pitchfork}$ having coordinates:
$$
S_{\Gamma(\tau)^\pitchfork}^{\bar{\Lambda}} (\Gamma (t)) = (S_t - S_\tau)^{-1} \, .
$$
The curve:
\[
\Gamma^{\overline{\Lambda}}(t)=(t-\tau)\left( \Gamma (t) - \overline{\Lambda} \right)_{\Gamma(\tau)^{\pitchfork}}
\]
is defined for $t\neq \tau$ in the vector space $\overrightarrow{\Gamma(\tau)}^{\pitchfork}$, where $\left( \Gamma (t) - \overline{\Lambda} \right)_{\Gamma(\tau)^{\pitchfork}}$ is the vector corresponding to the map
\[
\left( \Gamma (t) - \overline{\Lambda} \right)_{\Gamma(\tau)^{\pitchfork}}=\langle \Gamma (t) ,\overline{\Lambda},\Gamma(\tau) \rangle : \Gamma (t) \rightarrow \Gamma(\tau)
\]
with coordinates given by the function $F(t, \tau) = (t-\tau)  (S_t - S_\tau)^{-1}$.
By Lemma \ref{lem-derivative-3}, $F$ can be extended smoothly to $t=\tau$, because $F(t,t)=(S'_t)^{-1}$ and then $\Gamma^{\overline{\Lambda}}(t)$ can be smoothly extended to $t=\tau$. Moreover, for any $\Delta\in \Gamma(\tau)^{\pitchfork}$, we get:
\[
\left( \Gamma (t) - \overline{\Lambda} \right)_{\Gamma(\tau)^{\pitchfork}} = \left( \Gamma (t) - \Delta \right)_{\Gamma(\tau)^{\pitchfork}} + \left( \Delta - \overline{\Lambda} \right)_{\Gamma(\tau)^{\pitchfork}}
\]
and since the last term on the right hand side of the previous equation does not depend on $t$, then the curve 
\[
\Gamma^{\Delta} (t) = (t-\tau)\left( \Gamma (t) - \Delta \right)_{\Gamma(\tau)^{\pitchfork}}
\]
can also be extended to $t=\tau$.\\

Now, the second part of the proof, that is, the existence of $\Delta_\tau$, follows from the following Lemma.

\begin{lemma}\label{lem-derivative-4}
There exists a unique $\Delta_\tau \in \Gamma(\tau)^{\pitchfork}$ such that $(\Gamma^{\Delta_\tau})'(\tau)=0$, and if $\Delta_\tau \in \overline{\Lambda}^{\pitchfork}$, its  coordinates are given by: 
\[
S(\Delta_\tau) = S_{\tau}-2S'_{\tau} (S''_{\tau})^{-1} S'_{\tau}  .
\]
\end{lemma}

\begin{proof}
Fixed $\tau$, consider the Taylor series of the function $\widetilde{F}(t)=F(t,\tau)=(t-\tau)\left(S_t-S_{\tau}\right)^{-1}$ at $t=\tau$, which is given by:
\[
\widetilde{F}(t)=B_{-1}(\tau) + B_0(\tau)\cdot (t-\tau)+B_1(\tau)\cdot (t-\tau)^2+\cdots
\]
with $\widetilde{F}(\tau)=B_{-1}(\tau)=(S'_{\tau})^{-1}$, and $\widetilde{F}'(\tau)=B_0(\tau)$.
Notice that the coefficients $B_i(\tau)$ are symmetric matrices representing the coordinates $S^{\overline{\Lambda}}_{\Gamma(\tau)^{\pitchfork}}$ of the corresponding derivatives of $\Gamma^{\overline{\Lambda}} =\Gamma^{\overline{\Lambda}} (t)$, at $t=\tau$, in the affine space $\Gamma(\tau)^{\pitchfork}$, in particular $\widetilde{F}'(\tau)=B_0(\tau)\simeq (S^{\overline{\Lambda}}_{\Gamma(\tau)^\pitchfork})'(\tau)$. 

We can choose $\Delta_\tau \in\Gamma(\tau)^{\pitchfork}$ such that $\left(S(\Delta_\tau)-S_{\tau}\right)^{-1}=B_0(\tau)$, then the curve in $\overrightarrow{\Gamma(\tau)}^\pitchfork$:
\[
\Gamma^{\Delta}(t)=(t-\tau)\left( \Gamma (t) - \Delta_\tau \right)_{\Gamma(\tau)^{\pitchfork}}
\]
can be written in coordinates as: 
\[
\widetilde{F}_{\Delta}(t)=\widetilde{F}(t)-(t-\tau)B_0(\tau)
\]
verifying $\widetilde{F}'_{\Delta}(\tau)=0$, and corresponding to $(\Gamma^{\Delta})'(\tau ) = 0$.

Now, let us compute $B_0$. Notice that the function $f(t)=S_t-S_{\tau}$ has the Taylor series at $t=\tau$ given by:
\[
f(t)=S'_{\tau} (t-\tau)+\frac{1}{2} S''_{\tau} (t-\tau)^2+\frac{1}{6}S'''_{\tau} (t-\tau)^3 + \dots
\]
and so 
\[
\widetilde{G}(t)=G(t,\tau)=\frac{f(t)}{t-\tau}=S'_{\tau}+\frac{1}{2} S''_{\tau} (t-\tau)+\frac{1}{6}S'''_{\tau} (t-\tau)^2 + \dots
\]
therefore, denoting by $(i)$ the $i$--th derivative respect to $t$, we get 
\[
\widetilde{G}^{(i)}(\tau)=\frac{1}{i+1}S^{(i+1)}_{\tau}   .
\]

Since $\widetilde{F}(t)=\widetilde{G}(t)^{-1}$, deriving at $t=\tau$ we obtain 
\[
\widetilde{F}'(\tau)=-\widetilde{G}(\tau)^{-1}\widetilde{G}'(\tau) \widetilde{G}(\tau)^{-1} =-\frac{1}{2}(S'_{\tau})^{-1}S''_{\tau}(S'_{\tau})^{-1}
\] 
and because $\widetilde{F}'(\tau)=B_0(\tau)=\left(S(\Delta)-S_{\tau}\right)^{-1}$, then
\[
\left(S(\Delta)-S_{\tau}\right)^{-1}=-\widetilde{G}(\tau)^{-1}\widetilde{G}'(\tau) \widetilde{G}(\tau)^{-1} =-\frac{1}{2}(S'_{\tau})^{-1}S''_{\tau}(S'_{\tau})^{-1}
\] 
whence 
\[
S(\Delta) = S_{\tau} - 2 S'_{\tau}(S''_{\tau})^{-1} S'_{\tau}   .
\] 
\end{proof}

Now, we are ready to prove Thm. \ref{thm-curva-derivativa}.

\begin{proof}[Proof of Theorem \ref{thm-curva-derivativa}]  We can always choose $\overline{\Lambda}$ such that $\Delta_\tau \in \overline{\Lambda}^\pitchfork$, then since the coordinates of $\Delta_\tau$ in Lemma \ref{lem-derivative-4} coincide with the coordinates of $\Delta_\tau$ in Lemma \ref{lem-derivative-1}, that is:
\[
S(\Delta_\tau) = S_{\tau} - 2 S'_{\tau}(S''_{\tau})^{-1} S'_{\tau} = S^0_{\tau} \, ,
\]
then both  Lagrangian subspaces coincide, thus the proof is concluded.
\end{proof}


\begin{definition}\label{eq:derivative}
For a given Jacobi curve $\Gamma=\Gamma(\tau)$ the unique curve $\Delta \colon I\rightarrow \Gamma(\tau)^{\pitchfork}\subset \mathscr{L}(W)$ such that $\Delta (\tau)\in \Gamma (\tau)^\pitchfork$, and $\Gamma''_{\Delta(\tau)} (\tau) = 0$, for all $\tau \in I$, is called the \emph{derivative curve} of $\Gamma$.
\end{definition}


\subsection{The Ricci curvature of a Jacobi curve}

\subsubsection{The intrinsic definition of the Ricci curvature operator of a Jacobi curve}

We will consider as before $\Gamma (t)$ a Jacobi curve with $\Delta (\tau) = \Delta_\tau$ its derivative curve, \textit{cfr.} (\ref{eq:derivative}).  In the affine space $\Delta_\tau^\pitchfork$ we can take the origin at $\Gamma (\tau)$.   With that choice of origin, the linear space $\overrightarrow{\Delta_\tau^\pitchfork}$ associated to $\Delta_\tau^\pitchfork$ can be identified with (\textit{cfr.} \S \ref{sec:affine}, Eq. (\ref{eq:vector_space})):
$$
\overrightarrow{\Delta_\tau^\pitchfork} = \mathcal{S}(\Gamma (\tau), \Delta_\tau) = \left\{ S \colon \Gamma (\tau) \to \Delta_\tau \mid (S(u) \mid v ) = (S(v) \mid u), \forall u,v \in \Gamma (\tau)  \right\} \,.
$$
Moreover, as $\Gamma_{\Delta_\tau}'(\tau)$, and $\Gamma_{\Delta_\tau}'''(\tau)$ are in the vector space $\overrightarrow{\Delta_\tau^\pitchfork}$, and $\Gamma_{\Delta_\tau}'(\tau)$ is non-degenerate by definition of Jacobi curve, we can define the operator $R_\Gamma \colon \Gamma (\tau) \to \Gamma (\tau)$ as:
\begin{equation}\label{eq:RGamma}
R_\Gamma (\tau) =  (\Gamma_{\Delta_\tau}'(\tau))^{-1}  \circ \Gamma_{\Delta_\tau}'''(\tau) \, ,
\end{equation}
that is, $R_\Gamma (\tau)$ is the operator making the following diagram commutative:

\begin{equation}\label{diagrama-curvatura}
\begin{tikzpicture} [every node/.style={midway}]
\matrix[column sep={8em,between origins},
        row sep={3em}] at (0,0)
{ \node(Gam)   {$\Gamma\left(\tau\right)$}  ; & \node(At) {$\Delta_{\tau}$} ; \\
 ; &  \node(Gam2) {$\Gamma\left(\tau\right)$}  ;    \\};
\draw[->] (Gam) -- (At) node[anchor=south]  {$\Gamma'''_{\Delta_{\tau}}(\tau)$};
\draw[->] (At) -- (Gam2) node[anchor=west]  {$\left(\Gamma'_{\Delta_{\tau}} (\tau)\right)^{-1}$};
\draw[->] (Gam)   -- (Gam2) node[anchor=east] {$R_{\Gamma}\left(\tau\right)$\,\,\,};
\end{tikzpicture}
\end{equation}


\begin{definition}\label{def:Ricci}
Let $\Gamma (t)$ be a Jacobi curve. The curve of linear maps $R_\Gamma (\tau) \colon \Gamma (\tau) \to \Gamma (\tau)$ defined by the previous diagram, (\ref{diagrama-curvatura}), that is, $R_\Gamma (\tau) = \left(\Gamma'_{\Delta_\tau}\right)^{-1} \circ \Gamma'''_{\Delta_\tau}$,
will be called the Ricci curvature tensor of the curve.  We will say that the (regular) Jacobi curve $\Gamma (t)$ is admissible if $R_\Gamma (\tau)$ is invertible.
\end{definition}

If we choose a symplectic basis $(\mathbf{e}, \overline{\mathbf{e}})$ such that the Jacobi curve $\Gamma(t)$ has the coordinate representation $\Gamma(t) = \left[ \begin{array}{c} I \\ S_t \end{array} \right]_{(\mathbf{e}, \overline{\mathbf{e}})}$, then the matrix $\mathbb{S}(S_\tau)$ associated to the linear map $R_\Gamma (\tau) \colon \Gamma (\tau) \to \Gamma (\tau)$ with respect to the basis $\mathbf{e} + \overline{\mathbf{e}} S_t$, will not depend whether we use the symplectic basis $(\mathbf{e}, \overline{\mathbf{e}})$ or any conformal symplectic basis $(\mathbf{e}_\lambda, \overline{\mathbf{e}}_\lambda)$ in the class of $[\omega]$, because the Lagrangian subspace $\Gamma (\tau)$ and the affine structure of $\Delta_\tau^\pitchfork$ does not, \textit{cfr.} Sect. \ref{sec:basis}.   This important observation implies that the Ricci curvature $R_\Gamma$ depends just on the conformal symplectic structure of the space $W$.

The Ricci curvature operator $R_\Gamma$ defined above coincides with the curvature operator introduced by Agrachev and Zelenko in \cite{Ag02}.
As a consequence of the computations in Lemma \ref{lem-derivative-1} it is possible to obtain its matrix representation from $\widetilde{S}_t$.    Deriving again in (\ref{eq-S-prima-2}) and taking into account (\ref{eq-B-C-D-tau}), we can write:
\[
\widetilde{S}'''_t  = C (D_t)^{-1}\left[S'''_t + 3S''_t (D_t)^{-1}S'_t + 3S'_t (D_t)^{-1}S''_t+ 6S'_t (D_t)^{-1}S'_t(D_t)^{-1}S'_t\right]\cdot\left((D_t)^{-1}B_t + \mathrm{Id}\right) 
\]
and for $t=\tau$, because of (\ref{eq-B-C-D-tau}) and (\ref{eq-Stau-A}), we get :
\begin{align}
\widetilde{S}'''_{\tau} & = CC^{-1}\left[S'''_{\tau} + 3S''_{\tau} C^{-1}S'_{\tau} + 3S'_{\tau} C^{-1}S''_{\tau}+ 6S'_{\tau} C^{-1}S'_{\tau}C^{-1}S'_{\tau}\right]\mathrm{Id} = \nonumber \\
& = S'''_{\tau} - \frac{3}{2}S''_{\tau} (S'_{\tau})^{-1}S''_{\tau} -\frac{3}{2}S''_{\tau} (S'_{\tau})^{-1}S''_{\tau}+ \frac{6}{4}S''_{\tau} (S'_{\tau})^{-1}S''_{\tau} = \nonumber  \\
& = S'''_{\tau} - \frac{3}{2}S''_{\tau} (S'_{\tau})^{-1}S''_{\tau} =   S'_{\tau}\cdot\left[(S'_{\tau})^{-1} S'''_{\tau} - \frac{3}{2}(S'_{\tau})^{-1}S''_{\tau} (S'_{\tau})^{-1}S''_{\tau}  \right]  = S'_{\tau}\cdot\mathbb{S}\left(S_{\tau} \right) \, , \label{eq-curvatura-R}
\end{align}
with $\mathbb{S}\left(S_{\tau} \right)$ denoting the Schwarzian derivative of the curve $S_\tau$ (that will be discussed at length in the coming section).  The matrices $\widetilde{S}'''_{\tau}$, $\widetilde{S}'_{\tau}$ correspond, respectively, to the coordinates of the maps $\Gamma'''_{\Delta_\tau}$ and $\Gamma'_{\Delta_\tau}$ so, using (\ref{eq-Sprima-tau}) and (\ref{eq-curvatura-R}), the expression of $R_{\Gamma}$ in the same coordinates is given by the matrix $\mathbb{S}(S_{\tau})$:

\begin{equation}\label{eq:Ricci_Schwarz}
 \mathbb{S}\left(S_{\tau} \right) =  (S'_{\tau})^{-1} S'''_{\tau} - \frac{3}{2}\left( (S'_{\tau})^{-1}S''_{\tau} \right)^2\, ,
\end{equation}
in full agreement with the results in \cite{Ag02}.

\begin{definition} For any non--singular matrix $S_t=S(t)\in \mathbb{R}^{k\times k}$, depending smoothly on $t\in I\subset \mathbb{R}$, the expression found in (\ref{eq:Ricci_Schwarz}) will be called the Schwarzian derivative of $S_t$.
\end{definition}

\subsubsection{The coordinate representation of the Ricci curvature: The Schwarzian derivative and the change of parameter formula:}

We will study first how the Schwarzian derivative $\mathbb{S}(S_t)$ of $S_t$ behaves under a change of parameter. 

\begin{lemma}\label{lem:changeS}
Let us consider two smooth functions $\psi:I\subset \mathbb{R} \rightarrow J\subset \mathbb{R}$ and $\varphi:J\subset \mathbb{R} \rightarrow I\subset \mathbb{R}$, which are inverse functions of each other, that is, they satisfy $t=\psi(\overline{t})$, and $\overline{t}=\varphi(t)$.
Then we have:
\begin{equation}\label{eq-schwarz-cambio}
\left.\mathbb{S}\left( \psi \right)\right|_{\overline{t}} = -\left( \frac{d\psi}{d\overline{t}}  \right)^2  \left.\mathbb{S}\left( \varphi \right)\right|_{t=\psi(\overline{t})} \, , 
\end{equation}
with $\left.\mathbb{S}\left( \psi \right)\right|_{\overline{t}} = \frac{\psi'''(\overline{t})}{\psi'(\overline{t})}-\frac{3}{2}\left( \frac{\psi''(\overline{t})}{\psi'(\overline{t})} \right)^2$.
\end{lemma}

\begin{proof}
Since 
\[
\displaystyle{\frac{d\psi}{d\overline{t}}(\overline{t}) = \left( \frac{d\varphi}{dt}(\psi(\overline{t})) \right)^{-1}}
\]
or equivalently, simplifying the notation 
\[
\psi'(\overline{t}) = \left.\frac{1}{\varphi'\left(t\right)}\right|_{t=\psi(\overline{t})}  .
\]
Then we have
\[
\psi''(\overline{t}) = \left.\frac{-\varphi''\left(t\right)\cdot \psi'(\overline{t})}{\left(\varphi'\left(t\right)\right)^2}\right|_{t=\psi(\overline{t})}
\]
\begin{align*}
\psi'''(\overline{t}) & = \left.\frac{-\varphi'''\left(t\right)\cdot \left(\psi'(\overline{t})\right)^2 \varphi'\left(t\right)- \varphi''\left(t\right)\cdot \psi''(\overline{t})\varphi'\left(t\right)+2\left(\varphi''\left(t\right)\right)^2 \left(\psi'(\overline{t})\right)^2}{\left(\varphi'\left(t\right)\right)^3}\right|_{t=\psi(\overline{t})} = \\
& = \left.\frac{-\varphi'''\left(t\right)\cdot \left(\psi'(\overline{t})\right)^2 \varphi'\left(t\right)+3\left(\varphi''\left(t\right)\right)^2 \left(\psi'(\overline{t})\right)^2}{\left(\varphi'\left(t\right)\right)^3}\right|_{t=\psi(\overline{t})} 
\end{align*}
Now, we can compute 
\begin{align*}
\left.\mathbb{S}\left( \psi \right)\right|_{\overline{t}} & = \frac{\psi'''(\overline{t})}{\psi'(\overline{t})}-\frac{3}{2}\left( \frac{\psi''(\overline{t})}{\psi'(\overline{t})} \right)^2 = \\
& = \left.\left(\frac{-\varphi'''\left(t\right)\cdot \left(\psi'(\overline{t})\right)^2 \varphi'\left(t\right)+3\left(\varphi''\left(t\right)\right)^2 \left(\psi'(\overline{t})\right)^2}{\left(\varphi'\left(t\right)\right)^2}-\frac{3}{2}\left( \frac{-\varphi''\left(t)\right)\cdot \psi'(\overline{t})}{\varphi'\left(t\right)} \right)^2 \right)\right|_{t=\psi(\overline{t})}= \\
& = \left.\left(-\frac{\varphi'''\left(t\right)}{\varphi'\left(t\right)}\cdot \left(\psi'(\overline{t})\right)^2 +\frac{3}{2}\frac{\left(\varphi''\left(t\right)\right)^2}{\left(\varphi'\left(t\right)\right)^2}\cdot \left(\psi'(\overline{t})\right)^2 \right)\right|_{t=\psi(\overline{t})}= \\
& = -\left(\psi'(\overline{t})\right)^2 \cdot \left.\mathbb{S}\left( \varphi \right)\right|_{t=\psi(\overline{t})}
\end{align*}
as we wanted to show.
\end{proof}

\begin{proposition}\label{prop-cambio-curvatura}
Let $t=\psi(\overline{t})$ be a change of parameter for a Jacobi curve $\overline{\Gamma}$, that is $\overline{\Gamma}=\overline{\Gamma}(\overline{t})=\Gamma(\psi(\overline{t}))$ such that $\overline{\Gamma}(\overline{t})\simeq \begin{bmatrix} \mathrm{Id} \\ \overline{S}_{\overline{t}} \end{bmatrix}$ and $\Gamma(t)\simeq \begin{bmatrix} \mathrm{Id} \\ S_{t} \end{bmatrix}$, are the corresponding expressions in coordinates. Then 
\[
\mathbb{S}\left( \overline{S}_{\overline{t}} \right) = \left( \frac{d\psi}{d\overline{t}}  \right)^2 \mathbb{S}( S_{\psi(\overline{t})} ) + \left.\mathbb{S}\left( \psi \right)\right|_{\overline{t}}\cdot \mathbf{\mathrm{Id}}
\]
\end{proposition}

\begin{proof}
Since $\overline{S}_{\overline{t}}=S_{\psi(\overline{t})}$ then 
\begin{align}
\overline{S}'_{\overline{t}} & =\frac{d}{d\overline{t}}\overline{S}_{\overline{t}} = \psi'(\overline{t})S'_{\psi(\overline{t})}  \label{eq-cambio-1} \\
\overline{S}''_{\overline{t}} & =\frac{d^2}{d\overline{t}^2}\overline{S}_{\overline{t}} = \psi''(\overline{t})S'_{\psi(\overline{t})}+\left(\psi'(\overline{t})\right)^2 S''_{\psi(\overline{t})} \label{eq-cambio-2} \\
\overline{S}'''_{\overline{t}} & =\frac{d^3}{d\overline{t}^3}\overline{S}_{\overline{t}} = \psi'''(\overline{t})S'_{\psi(\overline{t})}+ 3\psi'(\overline{t})\psi''(\overline{t})S''_{\psi(\overline{t})}+\left(\psi'(\overline{t})\right)^3 S'''_{\psi(\overline{t})} \nonumber
\end{align}

So, we have
\begin{align*}
\mathbb{S}\left( \overline{S}_{\overline{t}} \right) &=  ( \overline{S}'_{\overline{t}} )^{-1} \overline{S}'''_{\overline{t}} - \frac{3}{2} \left( ( \overline{S}'_{\overline{t}} )^{-1} \overline{S}''_{\overline{t}} \right)^2 = \\
& = \left[ \frac{1}{\psi'(\overline{t})} ( S'_{\psi(\overline{t})} )^{-1} \left( \psi'''(\overline{t})S'_{\psi(\overline{t})}+ 3\psi'(\overline{t})\psi''(\overline{t})S''_{\psi(\overline{t})}+\left(\psi'(\overline{t})\right)^3 S'''_{\psi(\overline{t})}\right)\right] - \\
&- \frac{3}{2} \left[ \frac{1}{\psi'(\overline{t})} ( S'_{\psi(\overline{t})} )^{-1} \left(\psi''(\overline{t})S'_{\psi(\overline{t})}+\left(\psi'(\overline{t})\right)^2 S''_{\psi(\overline{t})} \right) \right]^2 = \\
& =  \left[ \frac{\psi'''(\overline{t})}{\psi'(\overline{t})} \cdot \mathbf{\mathrm{Id}} + 3\psi''(\overline{t})( S'_{\psi(\overline{t})} )^{-1}S''_{\psi(\overline{t})}+\left(\psi'(\overline{t})\right)^2 ( S'_{\psi(\overline{t})} )^{-1} S'''_{\psi(\overline{t})}\right] - \\
&- \frac{3}{2} \left[ \frac{\psi''(\overline{t})}{\psi'(\overline{t})} \cdot \mathbf{\mathrm{Id}} + \psi'(\overline{t})( S'_{\psi(\overline{t})} )^{-1} S''_{\psi(\overline{t})}  \right]^2 = \\
& = \left[ \frac{\psi'''(\overline{t})}{\psi'(\overline{t})} -\frac{3}{2}\left( \frac{\psi''(\overline{t})}{\psi'(\overline{t})}\right)^2 \right] \cdot \mathbf{\mathrm{Id}} + \left(\psi'(\overline{t})\right)^2 \left[ ( S'_{\psi(\overline{t})} )^{-1} S'''_{\psi(\overline{t})} - \frac{3}{2}\left( ( S'_{\psi(\overline{t})} )^{-1} S''_{\psi(\overline{t})}  \right)^2  \right]  = \\
&= \left.\mathbb{S}\left( \psi \right)\right|_{\overline{t}}\cdot \mathbf{\mathrm{Id}} + \left( \frac{d\psi}{d\overline{t}}  \right)^2 \mathbb{S}( S_{\psi(\overline{t})} )    
\end{align*}
as claimed.
\end{proof}

The result of Prop. \ref{prop-cambio-curvatura} tell us how the Ricci curvature $R_\Gamma (\tau)$ transforms under a change of parameter $t = \psi(\bar{t})$:
\begin{equation}\label{eq:curvatura-Rpar}
\boxed{
R_{\overline{\Gamma}}(\overline{t}) = \left( \frac{d\psi}{d\overline{t}}  \right)^2 \left.R_{\Gamma}\left(t\right)\right|_{\psi(\overline{t})} + \left.\mathbb{S}\left( \psi \right)\right|_{\overline{t}}\cdot \mathbf{\mathrm{Id}}
}
\end{equation}

\begin{remark}
Note that if $R_{\Gamma}(t)$ is diagonalisable, the same will be true for $R_{\overline{\Gamma}}(\overline{t})$.
\end{remark}

\subsection{The geometric arc parameter of a Jacobi curve}

\begin{definition} Given a Jacobi curve $\Gamma=\Gamma(t)\simeq\begin{bmatrix} I \\ S_t \end{bmatrix}$, we define its parametric Ricci curvature as the real valued function on the parameter $t$ of the Jacobi curve given by:
\[
\mathrm{Ric}_t = \mathrm{tr}\left( R_{\Gamma}(t) \right) = \mathrm{tr}\left( \mathbb{S}\left( S_t \right) \right)  .
\]
\end{definition}

\begin{definition}
Let $\Gamma (t)$ be a Jacobi curve and $\overline{\Gamma}(\overline{t}) = \Gamma(\psi (\overline{t}))$ a reparametrization of $\Gamma (t)$.  We will say that $\overline{t}$ is a projective parameter if 
\[
\overline{\mathrm{Ric}}_{\overline{t}} = \mathrm{tr}\left( R_{\overline{\Gamma}(\overline{t})} \right) \equiv 0   \, .
\]
\end{definition}

\begin{proposition}
Any Jacobi curve $\Gamma = \Gamma (t)$, has a projective parameter.
\end{proposition}

\begin{proof}
Given the Jacobi curve $\Gamma (t)$, consider the reparametrization $\bar{t} = \varphi (t)$, obtained solving the Schwarzian equation:
$$
\mathbb{S}(\varphi (t)) = \frac{1}{n} \mathrm{Ric}_\Gamma (t) \, ,
$$
then, because of Eq. (\ref{eq:curvatura-Rpar}) and Lem. \ref{lem:changeS}, Eq. (\ref{eq-schwarz-cambio}), we get immediately, that $\overline{\mathrm{Ric}}_{\bar{t}} = 0$.
\end{proof}

If $\bar{t}$ is a projective parameter for the Jacobi curve $\Gamma$, taking the trace in Eq. (\ref{eq:curvatura-Rpar}), we get: 
\[
0 = \left( \frac{d\psi}{d\overline{t}}  \right)^2 \mathrm{Ric}_{\psi(\overline{t})} + n \left.\mathbb{S}\left( \psi \right)\right|_{\overline{t}} \, ,
\]
whence:
\[
\left.\mathbb{S}\left( \psi \right)\right|_{\overline{t}} = -\frac{1}{n} \left( \frac{d\psi}{d\overline{t}}  \right)^2 \mathrm{Ric}_{\psi(\overline{t})} 
\]
and Eq. (\ref{eq:curvatura-Rpar}) becomes:
\begin{equation} \label{eq-curvatura-Rpar-2}
\boxed{
R_{\overline{\Gamma}}(\overline{t}) = \left( \frac{d\psi}{d\overline{t}}  \right)^2 \left[ R_{\Gamma}\left(\psi(\overline{t})\right) -\frac{1}{n} \mathrm{Ric}_{\psi(\overline{t})} \cdot \mathbf{\mathrm{Id}}\right]  .
}
\end{equation}

Given the projective parameter $\overline{t}$, we can define the geometric arc parameter $s_{\overline{\Gamma}}$, recall Sect. \ref{sec:geometric_arc}, as:
\[
ds_{\overline{\Gamma}} = \overline{\zeta}(\overline{t}) d\overline{t} \quad \text{ with } \quad \overline{\zeta}(\overline{t})=\sqrt[2n]{\left| \mathrm{det}\left(R_{\overline{\Gamma}}(\overline{t})\right) \right|} \, .
\]
Let us observe first, that the geometric arc parameter $s_{\overline{\Gamma}}$ given above is well defined and does not depend on the projective parameter $\bar{t}$ used to define it.  Indeed,  if $\bar{t}$, and $t = \psi (\bar{t})$, are projective parameters, then $\mathbb{S}(\psi) = 0$, and we get:
$$
\psi^* (ds_\Gamma) = \psi^*\left( \sqrt[2n]{\det R_\Gamma (t)} dt  \right) = \sqrt[2n]{\det R_\Gamma (\psi(\bar{t}))} \,\, \frac{d\psi}{d\bar{t}} d\bar{t}  \, ,
$$
but then, using (\ref{eq:curvatura-Rpar}), with $\mathbb{S}(\psi) = 0$ (the parameter $t$ is projective by assumption), we get: 
$$
\sqrt[2n]{\det R_\Gamma (\psi(\bar{t}))} \frac{d\psi}{d\bar{t}} d\bar{t} = \sqrt[2n]{\det \left(\frac{d\psi}{d\bar{t}}\right)^{-2} R_{\overline{\Gamma}} (\bar{t})} \, \,  \frac{d\psi}{d\bar{t}} d\bar{t} = \sqrt[2n]{\det R_{\overline{\Gamma}} (\bar{t})} d\bar{t}  = ds_{\overline{\Gamma}}\, , 
$$
with $\overline{\Gamma}(\bar{t}) = \Gamma (\psi (\bar{t}))$.

Then the geometric arc element depending on any parameter $t$ such that $t=\psi(\overline{t})$, where $\overline{t}=\varphi(t)$ is a projective parameter, can be computed easily\footnote{Without any lack of generality we may assume that the change of parameter preserves the orientation, that is, $\frac{d\psi}{d\overline{t}} >0$.}:
\begin{align*}
ds_{\Gamma} (t) & = \left. \varphi^{\ast}\left( ds_{\overline{\Gamma}} \right)\right|_{t} = \\
& = \overline{\zeta}(\varphi(t))~\frac{d\varphi}{dt}(t)~dt = \\
& = \sqrt[2n]{\left| \mathrm{det}\left(R_{\overline{\Gamma}}(\overline{t})\right) \right|} ~\frac{d\varphi}{dt}(t)~dt = \\
& = \sqrt[2n]{\left| \mathrm{det}\left(\left( \frac{d\psi}{d\overline{t}}(\varphi(t))  \right)^2 \left[ R_{\Gamma}(t) -\frac{1}{n} \mathrm{Ric}_{t} \cdot \mathbf{\mathrm{Id}}\right]\right) \right|} ~\frac{d\varphi}{dt}(t)~dt = \\
& = \sqrt[2n]{\left| \mathrm{det}\left( R_{\Gamma}(t) -\frac{1}{n} \mathrm{Ric}_{t} \cdot \mathbf{\mathrm{Id}}\right) \right|} ~\frac{d\psi}{d\overline{t}}(\varphi(t))~\frac{d\varphi}{dt}(t)~dt = \\
& = \sqrt[2n]{\left| \mathrm{det}\left( R_{\Gamma}(t) -\frac{1}{n} \mathrm{Ric}_{t} \cdot \mathbf{\mathrm{Id}}\right) \right|} ~dt 
\end{align*}
that is,
\begin{equation}\label{eq-zeta}
\boxed{ ds_{\Gamma} = \zeta(t) dt  \quad \text{ where } \quad \zeta(t)=\sqrt[2n]{\left| \mathrm{det}\left( R_{\Gamma}(t) -\frac{1}{n} \mathrm{Ric}_{t} \cdot \mathbf{\mathrm{Id}}\right) \right|}  \, \, . }
\end{equation}

Notice that, if $\overline{\Gamma}=\overline{\Gamma}(s)$ is parametrized by the geometric arc parameter $s$ such that $t=\psi(s)$ and $s=\varphi(t)$, then 
\[
\frac{d\psi}{d\overline{t}}(\varphi(t)) = \left( \frac{d\varphi}{dt}(t) \right)^{-1} = \frac{1}{\zeta(t)}
\]
and therefore, by the equations (\ref{eq-schwarz-cambio}) and (\ref{eq-curvatura-Rpar-2}), the expressions of the curvature operator and the Ricci curvature for a Jacobi curve parametrized by the geometric arc parameter $s$ can be written, in terms of any parameter $t$, as follows:
\begin{align}
& \boxed{  R_{\overline{\Gamma}}(\varphi(t)) = \frac{1}{\zeta^{2}(t)}  \left[ R_{\Gamma}(t) -\left.\mathbb{S}\left(\varphi\right)\right|_{t} \cdot \mathbf{\mathrm{Id}}\right]  } \label{eq-curv-arco} \\
& \boxed{ \overline{\mathrm{Ric}}_{\varphi(t)} =\frac{1}{\zeta^{2}(t)}  \left[ \mathrm{Ric}_{t} - n  \left.\mathbb{S}\left(\varphi\right)\right|_{t}\right] \label{eq-curv-ricci} } \\
& \boxed{ \mathbb{S}(\varphi) = \frac{\zeta''}{\zeta} - \frac{3}{2} \left( \frac{\zeta'}{\zeta}\right)^2\label{eq-curv-S} }
\end{align}
with $\zeta (t)$ given by (\ref{eq-zeta}). 

\begin{definition}\label{def:absolute_curvature}
With the previous conventions and notations, the operator given in Eq. (\ref{eq-curv-arco}), will be called the absolute curvature operator of the Jacobi curve $\Gamma (t)$ and it will be denoted as $\mathcal{R}_\Gamma(t)$, i.e., 
\begin{equation}\label{eq:absolute_curvature}
\mathcal{R}_\Gamma(t) = \frac{1}{\zeta^{2}(t)}  \left[ R_{\Gamma}(t) -\left.\mathbb{S}\left(\varphi\right)\right|_{t} \cdot \mathbf{\mathrm{Id}}\right] 
\end{equation}
\end{definition}

\begin{remark}\label{rmk-zeta_1}
Since $d s_{\overline{\Gamma}} = ds$, because of (\ref{eq-zeta}), we get:
\[
\left| \mathrm{det}\left( \mathcal{R}_\Gamma(s) -\frac{1}{n} \overline{\mathrm{Ric}}_{s} \cdot \mathbf{\mathrm{Id}}\right) \right| = 1 \, .
\]
\end{remark}


\subsection{Symmetry  of the Ricci curvature}

Because of Prop. \ref{prop-bilineal-tangente}, we identify the tangent vector $\Gamma'(t)$ of the curve $\Gamma$ with the corresponding symmetric bilinear form $\dot{\Gamma}(0)$, so we will denote $\Gamma'(t)\in \mathcal{S}\left(\Gamma(t)\right)$.
It will be helpful to see $\Gamma'(t)$ as a map 
\[
\Gamma'(t):\Gamma(t) \rightarrow \Gamma(t)^{\ast}
\] 
where $\Gamma(t)^{\ast}$ denotes the dual space of $\Gamma(t)$, because denoting by $\langle u,v \rangle_{\Gamma'(t)} \equiv\Gamma'(t)\left(u,v\right)$ the bilinear form $\Gamma'(t)$ then 
\[
\Gamma'(t)(u)= \langle u,\cdot \rangle_{\Gamma'(t)}\in\Gamma(t)^{\ast}  .
\]

Also observe that, since the derivative point $\Delta_\tau$ corresponding to $\Gamma(\tau)$ is transversal to $\Gamma(\tau)$, that is $\Delta_\tau\in \Gamma(\tau)^{\pitchfork}$, then the map $\Delta_\tau \rightarrow \Gamma(\tau)^{\ast}$ defined by $u\mapsto \left.\omega(u,\cdot)\right|_{\Gamma(\tau)} = (u\mid \cdot)_{\Gamma(\tau)}$, is a linear isomorphism, and we can write the bilinear form $\langle \cdot ,\cdot \rangle_{\Gamma'(t)}\in\Gamma(t)^{\ast}$ in terms of the symplectic form $\omega$ as (recall Eq. (\ref{eq:vector_space})):
\[
\langle u,v \rangle_{\Gamma'(t)} = \left(\Gamma'(t)(u) \mid v \right) = \left(\Gamma'(t)(v) \mid u \right)
\]
where we identify the form $\Gamma'(t)(u)\in\Gamma(t)^{\ast}$ with its corresponding vector in $\Delta_\tau$.

\begin{theorem}\label{thm:symmetric}
The Ricci curvature operator $R_\Gamma(\tau) :\Gamma(\tau)\rightarrow \Gamma(\tau)$ is symmetric with respect to the symmetric bilinear form $\Gamma'(\tau)$.
\end{theorem}

\begin{proof}
Let us take into account that the linear maps $\Gamma^{(i)}(\tau):\Gamma(\tau)\rightarrow \Delta_\tau$ given by the derivatives of the curve $\Gamma$ for $i\in\mathbb{N}$ verify that:
\[
\left( \Gamma^{(i)}(\tau)(u) \mid v \right) + \left( u \mid \Gamma^{(i)}(\tau)(v) \right)=0
\]
and, consequently:
\[
\left( \Gamma'''(\tau)(u) \mid  v \right) = \left(\Gamma'''(\tau)(v) \mid u \right)  .
\]
By diagram \ref{diagrama-curvatura}, we get:
\begin{align*}
\langle R_\Gamma(\tau) (u),v \rangle_{\Gamma'(t)} & = \left(\Gamma'(t)\left(\mathcal{R}_{\Gamma(\tau)}(u)\right) \mid v \right)=\left(\Gamma'''(\tau)(u)\mid v \right) = \\
& =  \left(\Gamma'''(\tau)(v) \mid u \right)= \langle R_\Gamma(\tau) (v),u \rangle_{\Gamma'(t)} = \\
& = \langle u,R_\Gamma(\tau) (v) \rangle_{\Gamma'(t)}
\end{align*}
as claimed.
\end{proof}

\begin{corollary}\label{cor-R-diagonal}
The Ricci curvature operator $R_\Gamma(\tau) :\Gamma(\tau)\rightarrow \Gamma(\tau)$ is diagonalizable, i.e. there exists a basis of eigenvectors $\mathbf{f}=\left( \mathbf{f}_1,\ldots ,\mathbf{f}_m \right)$ of $\Gamma(\tau)$.
\end{corollary}


\section{Cartan theory of Jacobi curves}\label{sec:Cartan}

This section will be devoted to the development of a Cartan-like theory of moving frames describing the structure of Jacobi curves.   

\subsection{Moving frames and the Cartan matrix of a Jacobi curve}\label{sec:moving_frames}

Let  $\left( \mathbf{e},\overline{\mathbf{e}} \right)=\left( \mathbf{e}_1,\ldots ,\mathbf{e}_m,\overline{\mathbf{e}}_1,\ldots,\overline{\mathbf{e}}_m \right)$ be a symplectic basis of $\left(W,\omega\right)$.   Now, consider a Jacobi curve $\Gamma=\Gamma(\tau)$ such that $\Gamma(\tau)\simeq \left[ \begin{matrix} I \\ S_{\tau} \end{matrix} \right] $ where $\tau$ is the geometric arc parameter. 
Let us assume that $\Gamma(\tau)$ is such that $S'_{\tau}$ is positive or negative definite. 
Then, by Corollary \ref{cor-R-diagonal}, there exists a basis $\mathbf{f}=\left( \mathbf{f}_1,\ldots ,\mathbf{f}_m \right)$ of eigenvectors of the curvature operator $R_\Gamma\left(\tau\right) :\Gamma\left(\tau\right) \rightarrow \Gamma\left(\tau\right)$ orthonormal with respect to $S'_{\tau}$.     By Lemma \ref{lema-base-simplectica}, there exists a complementary basis $\overline{\mathbf{f}}=\left(\overline{\mathbf{f}}_1,\ldots,\overline{\mathbf{f}}_m \right)$ of $\Delta_\tau\in\mathscr{L}(W)$ such that $\left(\mathbf{f},\overline{\mathbf{f}}\right)$ is symplectic basis of $W$. Then, in virtue of diagram (\ref{diagrama-curvatura}), we write, \textit{cfr.} (\ref{eq:sym_basis}):
\begin{equation}\label{eq-f-fbarra}
\left\{\begin{matrix} \mathbf{f}=\mathbf{e}~M_{\tau}+\overline{\mathbf{e}}~S_{\tau}~M_{\tau} \\ \overline{\mathbf{f}}=\mathbf{e}~\overline{M}_{\tau}+\overline{\mathbf{e}}~S^0_{\tau}~\overline{M}_{\tau} \end{matrix} \right.
\end{equation} 
where, recall (\ref{eq-base-simplectica}):
\begin{equation}\label{eq-Mbarra}
\overline{M}_{\tau}=(S^0_{\tau}-S_{\tau})^{-1}\left(M_{\tau}^{T}\right)^{-1}
\end{equation}
and $S^0_{\tau}$ is the derivative curve of $\Gamma(s)$ given by equation (\ref{eq-A-tau}). 
Therefore:
\begin{equation}\label{eq:e-f}
\left(\mathbf{f},\overline{\mathbf{f}}\right)=\left(\mathbf{e},\overline{\mathbf{e}}\right)\mathbf{P}_{\tau}
\end{equation}
where $\mathbf{P}_{\tau}\in \mathbb{R}^{2m\times 2m}$ has the form: 
\begin{equation}\label{eq-matriz-P}
\mathbf{P}_{\tau} = \begin{pmatrix}
M_{\tau} & \overline{M}_{\tau} \\
S_{\tau} M_{\tau} & S^0_{\tau} \overline{M}_{\tau}
\end{pmatrix}   \, ,
\end{equation}
and we will call the symplectic basis $F = (\mathbf{f}, \overline{\mathbf{f}})$ a symplectic Frenet basis for $\Gamma$.

Deriving the equation (\ref{eq:e-f}) with respect to $\tau$, we get:
\begin{equation}\label{eq-matriz-Pprima}
\left(\mathbf{f}',\overline{\mathbf{f}}'\right)=\left(\mathbf{e},\overline{\mathbf{e}}\right)\mathbf{P}'_{\tau} = \left(\mathbf{f},\overline{\mathbf{f}}\right)\mathbf{P}^{-1}_{\tau} \mathbf{P}'_{\tau}
\end{equation}
where $\mathbf{f}' = d \mathbf{f} /d\tau$, $\overline{\mathbf{f}}' = d \overline{\mathbf{f}} /d \tau$, $\mathbf{P}'_{\tau} = d \mathbf{P}_{\tau} / d\tau $, and:
\begin{equation}\label{eq:Cartan_gen}
\mathbf{C}_{\tau} = \mathbf{P}^{-1}_{\tau} \mathbf{P}'_{\tau}   
\end{equation}
is the \emph{Cartan matrix} of the moving frame $\left(\mathbf{f},\overline{\mathbf{f}}\right)$.

\subsection{The structure of the Cartan matrix}

Observe that the matrix (\ref{eq-matriz-P}) can be factorized as 
\[
\mathbf{P}_{\tau} = \mathbf{U}_{\tau} \mathbf{V}_{\tau}
\]
where
\[
\mathbf{U}_{\tau} = \begin{pmatrix}
I & I \\
S_{\tau}  & S^0_{\tau} 
\end{pmatrix}   \quad \text{ and } \qquad 
\mathbf{V}_{\tau} = \begin{pmatrix}
M_{\tau} & 0 \\
0  & \overline{M}_{\tau} 
\end{pmatrix}   .
\] 

The inverse of $\mathbf{P}_{\tau}$ is 
\[
\mathbf{P}^{-1}_{\tau}= \mathbf{V}^{-1}_{\tau} \mathbf{U}^{-1}_{\tau} = \begin{pmatrix}
M^{-1}_{\tau} & 0 \\
0  & \overline{M}^{-1} _{\tau}
\end{pmatrix}
\cdot \begin{pmatrix}
(S^0_{\tau}-S_{\tau} )^{-1}S^0_{\tau} & -(S^0_{\tau}-S_{\tau} )^{-1} \\
-(S^0_{\tau}-S_{\tau} )^{-1}S_{\tau} & (S^0_{\tau}-S_{\tau} )^{-1}
\end{pmatrix} 
\] 
and its derivative with respect to $\tau$ is 
\[
\mathbf{P}'_{\tau}= \mathbf{U}'_{\tau} \mathbf{V}_{\tau}+ \mathbf{U}_{\tau} \mathbf{V}'_{\tau}
\] 
whence 
\begin{equation}\label{eq-matriz-C}
\boxed{ \mathbf{C}_{\tau}= \mathbf{P}^{-1}_{\tau}\mathbf{P}'_{\tau}= \mathbf{V}^{-1}_{\tau} \mathbf{U}^{-1}_{\tau}\mathbf{U}'_{\tau} \mathbf{V}_{\tau} + \mathbf{V}^{-1}_{\tau} \mathbf{V}'_{\tau}   } \, .
\end{equation}

Let us compute the expression of $\mathbf{C}_{\tau}$ depending on $S_{\tau}$ and its derivatives. 

By equation (\ref{eq-Mbarra}), we have that 
\begin{align}
\overline{M}_{\tau}^{-1} & = M_{\tau}^{T}(S^0_{\tau}-S_{\tau}) \label{eq-Mbarra-1}\\
\overline{M}'_{\tau} & = -(S^0_{\tau}-S_{\tau})^{-1}\left[ (S^0_{\tau}-S_{\tau})'(S^0_{\tau}-S_{\tau})^{-1} + (M^{T}_{\tau})^{-1}(M^{T}_{\tau})' \right] (M^{T}_{\tau})^{-1} \nonumber \\
\overline{M}_{\tau}^{-1} \overline{M}'_{\tau} & = - M_{\tau}^{T}(S^0_{\tau}-S_{\tau})'(S^0_{\tau}-S_{\tau})^{-1}(M^{T}_{\tau})^{-1} - (M^{T}_{\tau})'(M^{T}_{\tau})^{-1}  \label{eq-Mbarra-2} .
\end{align}

Moreover, since
\begin{align*}
\mathbf{U}^{-1}_{\tau}\mathbf{U}'_{\tau}& = 
\begin{pmatrix}
(S^0_{\tau}-S_{\tau} )^{-1}S^0_{\tau} & -(S^0_{\tau}-S_{\tau} )^{-1} \\
-(S^0_{\tau}-S_{\tau} )^{-1}S_{\tau} & (S^0_{\tau}-S_{\tau} )^{-1}
\end{pmatrix}
 \begin{pmatrix}
0 & 0 \\
S'_{\tau}  & (S^0_{\tau})'
\end{pmatrix} = \\
& = \begin{pmatrix}
-(S^0_{\tau}-S_{\tau} )^{-1}S'_{\tau} & -(S^0_{\tau}-S_{\tau} )^{-1}(S^0_{\tau})' \\
(S^0_{\tau}-S_{\tau} )^{-1}S'_{\tau} & (S^0_{\tau}-S_{\tau} )^{-1} (S^0_{\tau})'
\end{pmatrix} 
\end{align*}
then, using equations (\ref{eq-Mbarra}), (\ref{eq-Mbarra-1}) and (\ref{eq-Mbarra-2}), the two terms on the r.h.s. of (\ref{eq-matriz-C}) can be written as: 
\begin{align*}
\mathbf{V}^{-1}_{\tau} \mathbf{U}^{-1}_{\tau}\mathbf{U}'_{\tau} \mathbf{V}_{\tau}& = 
\begin{pmatrix}
M^{-1}_{\tau} & 0 \\
0  & \overline{M}^{-1}_{\tau}
\end{pmatrix}
\begin{pmatrix}
-(S^0_{\tau}-S_{\tau} )^{-1}S'_{\tau} & -(S^0_{\tau}-S_{\tau} )^{-1}(S^0_{\tau})' \\
(S^0_{\tau}-S_{\tau} )^{-1}S'_{\tau} & (S^0_{\tau}-S_{\tau} )^{-1} (S^0_{\tau})'
\end{pmatrix} 
\begin{pmatrix}
M_{\tau} & 0 \\
0  & \overline{M}_{\tau}
\end{pmatrix} = \\
& = \begin{pmatrix}
-M^{-1}_{\tau}(S^0_{\tau}-S_{\tau} )^{-1}S'_{\tau}M_{\tau} & -M^{-1}_{\tau}(S^0_{\tau}-S_{\tau} )^{-1}(S^0_{\tau})'(S^0_{\tau}-S_{\tau} )^{-1} (M^{T}_{\tau})^{-1} \\
M^{T}_{\tau}S'_{\tau}M_{\tau} & M^{T}_{\tau}(S^0_{\tau})'(S^0_{\tau}-S_{\tau} )^{-1}(M^{T}_{\tau})^{-1}
\end{pmatrix} \\
\mathbf{V}^{-1}_{\tau} \mathbf{V}'_{\tau}& = 
\begin{pmatrix}
M^{-1}_{\tau}M'_{\tau} & 0 \\
0  & \overline{M}^{-1}_{\tau}\overline{M}'_{\tau}
\end{pmatrix} = \\
& = \begin{pmatrix}
M^{-1}_{\tau}M'_{\tau} & 0 \\
0  & - M_{\tau}^{T}(S^0_{\tau}-S_{\tau})'(S^0_{\tau}-S_{\tau})^{-1}(M^{T}_{\tau})^{-1} - (M^{T}_{\tau})'(M^{T}_{\tau})^{-1}
\end{pmatrix}
\end{align*}

From equations (\ref{eq-Stau-A}) and (\ref{eq-A-tau}), we can compute
\begin{align*}
(S^0_{\tau})' &= -3S'_{\tau}+2S'_{\tau}(S''_{\tau})^{-1}S'''_{\tau}(S''_{\tau})^{-1}S'_{\tau} \\
(S^0_{\tau}-S_{\tau} )^{-1}S'_{\tau}  & = -\frac{1}{2}(S'_{\tau})^{-1}S''_{\tau} \\
S'_{\tau}(S^0_{\tau}-S_{\tau} )^{-1}  & = -\frac{1}{2}S''_{\tau}(S'_{\tau})^{-1} \\
(S^0_{\tau})' (S^0_{\tau}-S_{\tau} )^{-1}& =  \left( -3S'_{\tau}+2S'_{\tau}(S''_{\tau})^{-1}S'''_{\tau}(S''_{\tau})^{-1}S'_{\tau}\right)\left(-\frac{1}{2}(S'_{\tau})^{-1}S''_{\tau} (S'_{\tau})^{-1}\right) = \\
& = \frac{3}{2}S''_{\tau}(S'_{\tau})^{-1} - S'_{\tau}(S''_{\tau})^{-1}S'''_{\tau}(S'_{\tau})^{-1} = \\
& = -S'_{\tau}(S''_{\tau})^{-1}S'_{\tau} \left( -\frac{3}{2}(S'_{\tau})^{-1}S''_{\tau}(S'_{\tau})^{-1}S''_{\tau} + (S'_{\tau})^{-1}S'''_{\tau}\right)(S'_{\tau})^{-1} = \\
& = -S'_{\tau}(S''_{\tau})^{-1}S'_{\tau}\cdot \mathbb{S}\left(S_{\tau}\right)\cdot (S'_{\tau})^{-1}
\end{align*}
So, using equation (\ref{eq-matriz-C}), we obtain the explicit expression of the Cartan matrix $\mathbf{C}_{\tau}$ in terms of $S_\tau$ and its derivatives:
\begin{equation}\label{eq-matriz-C-1}
 \mathbf{C}_{\tau}= 
\begin{pmatrix}
M^{-1}_{\tau} & 0 \\
0  & M^{T}_{\tau}
\end{pmatrix}
\begin{pmatrix}
\frac{1}{2}(S'_{\tau})^{-1}S''_{\tau} & -\frac{1}{2}\mathbb{S}\left(S_{\tau}\right) (S'_{\tau})^{-1} \\
S'_{\tau} & -\frac{1}{2}S''_{\tau}(S'_{\tau})^{-1}
\end{pmatrix}
\begin{pmatrix}
M_{\tau} & 0 \\
0  & (M^{T}_{\tau})^{-1}
\end{pmatrix}
+ 
\begin{pmatrix}
M^{-1}_{\tau}M'_{\tau} & 0 \\
0  & -\left(M^{-1}_{\tau}M'_{\tau}\right)^{T}
\end{pmatrix}  
\end{equation}
where we have used that $(M^{T}_{\tau})'(M^{T}_{\tau})^{-1}=\left(M^{-1}_{\tau}M'_{\tau}\right)^{T}$.

If we denote 
\[
 \mathbf{C}_{\tau}= 
\begin{pmatrix}
K_{11} & K_{12} \\
K_{21}  & -K_{11}^{T}
\end{pmatrix}
\]
then we have  
\begin{align*}
K_{11} &= \frac{1}{2}M^{-1}_{\tau}(S'_{\tau})^{-1}S''_{\tau}M_{\tau} + M^{-1}_{\tau}M'_{\tau} \\
K_{12} &= -\frac{1}{2} M^{-1}_{\tau} \mathbb{S}\left(S_{\tau}\right) (S'_{\tau})^{-1} (M^{T}_{\tau})^{-1} \\
K_{21} &= M^{T}_{\tau}S'_{\tau}M_{\tau}
\end{align*}

Because the column vectors of $M_{\tau}$ are normalized with respect to the bilinear form $S'_{\tau}$, see Sect. \ref{sec:moving_frames}, then 
\begin{equation}\label{eq-K21}
K_{21} = M^{T}_{\tau}S'_{\tau}M_{\tau} = \mathrm{Id}
\end{equation}
and therefore 
\begin{equation}\label{eq-SprimaMM}
\boxed{(S'_{\tau})^{-1} = M_{\tau} M^{T}_{\tau} .}
\end{equation}

Deriving the equation $M^{T}_{\tau}S'_{\tau}M_{\tau} = \mathrm{Id}$, we have
\[
(M^{T}_{\tau})' S'_{\tau} M_{\tau} + M^{T}_{\tau} S''_{\tau} M_{\tau} + M^{T}_{\tau} S'_{\tau} M_{\tau} ' =0
\]
and because of Eq. (\ref{eq-SprimaMM}), we get 
\[
M^{T}_{\tau} S''_{\tau} M_{\tau} = - \left[ M_{\tau}^{-1} M'_{\tau} +\left(M_{\tau}^{-1} M'_{\tau}\right)^{T} \right]  .
\]

So, substituting in the expressions of $K_{12}$ and $K_{11}$, we obtain 
\[
K_{12} = -\frac{1}{2} M^{-1}_{\tau} \mathbb{S}\left(S_{\tau}\right) (S'_{\tau})^{-1} (M^{T}_{\tau})^{-1} = -\frac{1}{2} M^{-1}_{\tau} \mathbb{S}\left(S_{\tau}\right) M_{\tau} 
\]
\begin{align*}
K_{11} & = \frac{1}{2} M^{-1}_{\tau}(S'_{\tau})^{-1}S''_{\tau}M_{\tau} + M^{-1}_{\tau}M'_{\tau} =  \\
& =  \frac{1}{2} M^{T}_{\tau}S''_{\tau}M_{\tau} + M^{-1}_{\tau}M'_{\tau} =  \\
& = -\frac{1}{2} \left[ M_{\tau}^{-1} M'_{\tau} +\left(M_{\tau}^{-1} M'_{\tau}\right)^{T} \right] + M^{-1}_{\tau}M'_{\tau} = \\
& = \frac{1}{2} \left[M_{\tau}^{-1} M'_{\tau} - \left(M_{\tau}^{-1} M'_{\tau}\right)^{T} \right] 
\end{align*}

Then, we can write the Cartan matrix as
\begin{equation}\label{eq-Cartan-tau}
\boxed{ \mathbf{C}_{\tau}= 
\begin{pmatrix}
\frac{1}{2} \left[M_{\tau}^{-1} M'_{\tau} - \left(M_{\tau}^{-1} M'_{\tau}\right)^{T} \right] & -\frac{1}{2} M^{-1}_{\tau} \mathbb{S}\left(S_{\tau}\right) M_{\tau}  \\
\mathrm{Id} & \frac{1}{2} \left[M_{\tau}^{-1} M'_{\tau} - \left(M_{\tau}^{-1} M'_{\tau}\right)^{T} \right]
\end{pmatrix}}
\end{equation}

It is important to observe that, since $M_{\tau}$ is the matrix of eigenvectors of $\mathbb{S}\left(S_{\tau}\right)$, then:
\[
K_{12}=-\frac{1}{2} M^{-1}_{\tau} \mathbb{S}\left(S_{\tau}\right) M_{\tau}=-\frac{1}{2} D_{\tau}
\]
is diagonal, and $D_{\tau}$ is the matrix of eigenvalues of the curvature tensor. Moreover, we also note that:
\[
K_{11}^T=\frac{1}{2} \left[\left(M_{\tau}^{-1} M'_{\tau}\right)^{T} - M_{\tau}^{-1} M'_{\tau} \right] =-K_{11} .
\]

Therefore the Cartan matrix has the block structure: 
\[
\boxed{ \mathbf{C}_{\tau}= 
\begin{pmatrix}
\Sigma_{\tau} & K_{\tau} \\
\mathrm{Id} & \Sigma_{\tau}
\end{pmatrix}  }
\]
where $K_{\tau}= -\frac{1}{2} D_{\tau}\in\mathbb{R}^{n\times n}$ is diagonal and $\Sigma_{\tau}\in\mathbb{R}^{n\times n}=\frac{1}{2} \left[\left(M_{\tau}^{-1} M'_{\tau}\right)^{T} - M_{\tau}^{-1} M'_{\tau} \right]$ skew--symmetric.

But the eigenvalues $k_i$ with $i=1,\ldots,n$ of $D_{\tau}$ still verify an additional property. Indeed, since $\Gamma=\Gamma(\tau)$ is parametrized by the geometric arc parameter $\tau$, because of Remark \ref{rmk-zeta_1}, we get:
\begin{equation}\label{eq-condicion-curvatura}
1 = \left| \mathrm{det}\left(\mathcal{R}_{\Gamma(\tau)}- \frac{1}{n}\mathrm{tr}\left(\mathcal{R}_{\Gamma(\tau)}\right)\mathrm{Id}\right) \right| = 
 \left| \mathrm{det}\left(D_{\tau}- \frac{1}{n}\mathrm{tr}\left(D_{\tau}\right)\mathrm{Id}\right) \right| =
\left| \prod_{i=1}^{m}\left(k_i- \frac{1}{n}\sum_{j=1}^{m}k_j\right)\right|
\end{equation}
that establishes a functional dependence among the functions $k_i$, $i=1,\ldots,n$.  Writing $\frac{1}{n}\sum_{j=1}^{m}k_j$ as $\bar{k}$, the previous condition (\ref{eq-condicion-curvatura}) becomes $\prod_{i=1}^{m} | k_i-  \bar{k} | = 1$, or even better:
$$
\prod_{i=1}^{m} \Delta k_i = 1 \, ,
$$
where $\Delta k_i = | k_i-  \bar{k} |$.
\\


Next, let us compute the matrix $\mathbf{C}_{\tau}$ for any parameter.
Let $t$ be another parameter for $\Gamma$ such that $t=\psi(\tau)$ and $\tau=\varphi(t)$ in such a way $\widetilde{\Gamma}=\widetilde{\Gamma}(t)=\Gamma\left( \varphi(t)\right)\simeq\begin{bmatrix} I \\ \widetilde{S}_t \end{bmatrix}$. 
Then we have:
\[
\psi'(\tau)=\frac{1}{\varphi'\left(\psi(\tau)\right)} = \frac{1}{\zeta\left(\psi(\tau)\right)} \, , \quad 
\psi''(\tau)=\frac{-\zeta'\left(\psi(\tau)\right)\cdot \psi'(\tau)}{\zeta^2\left(\psi(\tau)\right)} = \frac{-\zeta'\left(\psi(\tau)\right)}{\zeta^3\left(\psi(\tau)\right)}  \, .
\]
and, because of Prop. \ref{prop-cambio-curvatura}, and Lem. \ref{lem-derivative-4}, we get:
\begin{equation}\label{eq:Stau}
\mathbb{S}\left(S_{\tau}\right) = \frac{1}{\zeta^2(t)}\left(\mathbb{S}(\widetilde{S}_{t})-\mathbb{S}(\varphi)\cdot I \right)  \, .
\end{equation}

Because of equation (\ref{eq-curvatura-Rpar-2}), any eigenvector $\mathbf{f}_a$ of $\mathcal{R}_{\Gamma}(\tau) = \mathbb{S}\left(S_{\tau}\right)$  is also eigenvector of $R_{\widetilde{\Gamma}}(t) = \mathbb{S}(\widetilde{S}_{t})$. If $\mathbf{f}_a$ is normalized by $S'_{\tau}$ then $\widetilde{\mathbf{f}}_a = \frac{1}{\sqrt{\zeta(t)}}\mathbf{f}_a$ is normalized with respect to $\widetilde{S}'_{t}$.
Hence, we have that 
\[
\widetilde{M}_t^T\widetilde{S}'_{t}\widetilde{M}_t = \mathrm{Id} \quad \Leftrightarrow \quad \widetilde{S}'_{t} = \left(\widetilde{M}_t \widetilde{M}_t^T\right)^{-1}
\]
and $\widetilde{M}_t^{-1}\mathbb{S}(\widetilde{S}_{t})\widetilde{M}_t = \widetilde{D}_{t}$, is diagonal.  Moreover
\[
\left.M_{\tau}\right|_{\tau=\varphi(t)}= \sqrt{\zeta(t)}\widetilde{M}_t
\]
where $\widetilde{M}_t$ is the matrix of coordinates of the eigenvectors of $R_{\widetilde{\Gamma}}(t) $ orthonormal with respect to $\widetilde{S}'_{t}$. So, we have: 
\[
\left.\frac{d M_{\tau}}{dt}\right|_{\tau=\varphi(t)}\varphi'(t)= \frac{\zeta'(t)}{2\zeta^{1/2}(t)}\widetilde{M}_t + \zeta^{1/2}(t)\widetilde{M}'_t \, ,
\]
and since $\varphi'=\zeta$, then:
\[
\left.\frac{d M_{\tau}}{dt}\right|_{\tau=\varphi(t)} =  \frac{\zeta'(t)}{2\zeta^{3/2}(t)}\widetilde{M}_t + \frac{1}{\zeta^{1/2}(t)}\widetilde{M}'_t  \,  .
\]
Substituting the previous expressions in the blocks of the matrix (\ref{eq-Cartan-tau}) for $\tau=\varphi(t)$, we obtain:
\begin{equation}
\frac{1}{2} \left[M_{\tau}^{-1} M'_{\tau} - \left(M_{\tau}^{-1} M'_{\tau}\right)^{T} \right] = \frac{1}{2\zeta(t)} \left[ \widetilde{M}^{-1}_t \widetilde{M}'_{t} - \left(\widetilde{M}^{-1}_t \widetilde{M}'_{t}\right)^{T} \right] \, ,
\end{equation}
and also, using (\ref{eq:Stau}), we get:
\begin{equation}
\frac{1}{2} M^{-1}_{\tau} \mathbb{S}\left(S_{\tau}\right) M_{\tau}  =\frac{1}{\zeta(t)^2}\cdot \frac{1}{2}\left( \widetilde{M}^{-1}_{t}\mathbb{S}(\widetilde{S}_{t})\widetilde{M}_{t} - \mathbb{S}(\varphi)\cdot I \right) =\frac{1}{\zeta(t)^2}\cdot \frac{1}{2}\left( \widetilde{D}_{t} - \mathbb{S}(\varphi)\cdot I \right)  \, .
\end{equation}

Then equation (\ref{eq-matriz-C-1}) becomes:
\begin{empheq}[box=\widefbox]{align}
\mathbf{C}_{\varphi(t)} & =
\begin{pmatrix}
\frac{1}{2\zeta(t)} \left[ \widetilde{M}^{-1}_t \widetilde{M}'_{t} - \left(\widetilde{M}^{-1}_t \widetilde{M}'_{t}\right)^{T} \right] & -\frac{1}{2\zeta(t)^2}\left(\widetilde{D}_{t} - \mathbb{S}(\varphi)\cdot \mathrm{Id}\right) \\
\mathrm{Id} & \frac{1}{2\zeta(t)} \left[ \widetilde{M}^{-1}_t \widetilde{M}'_{t} - \left(\widetilde{M}^{-1}_t \widetilde{M}'_{t}\right)^{T} \right]
\end{pmatrix} 
\label{eq-matriz-C-2}
\end{empheq}

\begin{remark}
Note that if $\{k_i\}_{i=1,\ldots,n}$ are the eigenvalues of $\mathbb{S}(S_{\tau})$ and $\{\lambda_i\}_{i=1,\ldots,n}$ the corresponding ones for $\mathbb{S}(\widetilde{S}_{t})$ then, they are related by: 
\[
\boxed{ k_i(\varphi(t))= \frac{1}{\zeta^2(t)}\left(\lambda_i(t) - \mathbb{S}(\varphi)\right) }
\]
Indeed,
\begin{align*}
& ~\mathrm{det}\left( \mathbb{S}(S_{\tau}) - k_i \mathrm{Id}  \right) = 0
\Rightarrow ~\mathrm{det}\left(\frac{1}{\zeta^2} \left(\mathbb{S}(\widetilde{S}_{t})-\mathbb{S}(\varphi)\mathrm{Id}\right) - k_i \mathrm{Id}  \right) = 0 \Rightarrow \\
\Rightarrow & ~\mathrm{det}\left(\frac{1}{\zeta^2} \left[\mathbb{S}(\widetilde{S}_{t})-\left(\mathbb{S}(\varphi)+\zeta^2 k_i \right)\mathrm{Id} \right] \right) = 0
\Rightarrow ~\mathrm{det} \left(\mathbb{S}(\widetilde{S}_{t})-\left(\mathbb{S}(\varphi)+\zeta^2 k_i \right)\mathrm{Id}  \right) = 0 
\Rightarrow  ~\lambda_i = \mathbb{S}(\varphi)+\zeta^2 k_i \, .
\end{align*}
\end{remark}

\begin{remark}\label{rem:derivative_nongeom}  Observe that, with the notation above, and denoting by $\widetilde{S}_t^0 = S_\tau^0\mid_{\varphi (t)}$, the coordinates of the derivative curve depending on the non-geometric arc parameter $t$, we have that:
$$
\widetilde{S}_t^0 = \widetilde{S}_t - 2 \widetilde{S}_t' \left( \widetilde{S}_t'' - \frac{\zeta'}{\zeta} \widetilde{S}_t'\right)^{-1} \widetilde{S}_t' \, .
$$
\end{remark}


\section{A reconstruction theorem for Jacobi curves}\label{sec:reconstruction}

\subsection{Reconstruction of a Jacobi curve from its curvatures}

In this section we will show that the matrices $\Sigma=\Sigma(\tau)$ and $K = K(\tau)$ characterize the Jacobi curves for given initial values. 

Recall that we have fixed a symplectic basis $(\mathbf{e} ,\overline{\mathbf{e}})\subset W$ in which the matrix of the symplectic $2$--form $\omega$  is written by 
\[
J=\begin{pmatrix}
0 & \mathrm{Id} \\
-\mathrm{Id} & 0
\end{pmatrix} \in \mathbb{R}^{2n \times 2n}
\]
where $\mathrm{Id}\in \mathbb{R}^{n \times n}$ is the identity matrix.

First, we will state a lemma.

\begin{lemma}\label{lem-symplectic-f}
Let $\Sigma=\Sigma(\tau)\in\mathbb{R}^{n \times n}$ be a smooth skew--symmetric matrix and $K = K(\tau)\in\mathbb{R}^{n \times n}$ a smooth diagonal one. For a given symplectic basis $F_0=(\mathbf{f}_0 ,\overline{\mathbf{f}}_0)\subset W$, the solution $F(\tau)=(\mathbf{f}(\tau),\overline{\mathbf{f}}(\tau))$ for $\tau\in I\subset \mathbb{R}$ of the initial value problem:  
\begin{equation}\label{eq-EDO-Cartan-2}
\left\{
\begin{tabular}{l}
$\frac{d\left( \mathbf{f}, \overline{\mathbf{f}}\right) }{d\tau} = ( \mathbf{f}', \bar{\mathbf{f}}' ) = \left( \mathbf{f}, \overline{\mathbf{f}}\right)\begin{pmatrix}
\Sigma & K \\
\mathrm{Id} & \Sigma
\end{pmatrix}$ \\
$\mathbf{f}(0)=\mathbf{e}A_0 + \overline{\mathbf{e}} B_0$\\
$\overline{\mathbf{f}}(0)=\mathbf{e}\overline{A}_0 + \overline{\mathbf{e}} \overline{B}_0$
\end{tabular}
\right.
\end{equation}
is a symplectic basis for all $\tau \in I$ and all solutions have the form $PF$ where $P$ is a conformal symplectic transform, $P \in CSp$.
\end{lemma}

\begin{proof}
The existence and uniqueness of the solution $F=F(\tau) = (\mathbf{f}(\tau),\overline{\mathbf{f}}(\tau))$, is ensured by Picard-Lindel\"{o}f theorem \cite[Thm. 1.1]{Ha64}. So, we have that $\mathbf{f}=\left(\mathbf{f}_1,\ldots ,\mathbf{f}_m\right)$ and $\overline{\mathbf{f}}=\left(\overline{\mathbf{f}}_1,\ldots ,\overline{\mathbf{f}}_m\right)$, are the unique solution of the system (\ref{eq-EDO-Cartan-2}) that we write as 
\begin{equation}\label{eq-coordinates-ff}
\left\{
\begin{tabular}{l}
$\mathbf{f}=\mathbf{e}A + \overline{\mathbf{e}} B$ \\
$\overline{\mathbf{f}}=\mathbf{e}\overline{A} + \overline{\mathbf{e}} \overline{B}$
\end{tabular}
\right.
\end{equation}

Now, let us denote the $m\times m$ matrices
\begin{equation}\label{eq-matrices-omegaff}
\omega\left( \mathbf{f}, \mathbf{f} \right) = \left( A^{T}, B^{T}\right) J \begin{pmatrix} A \\ B \end{pmatrix} , \quad 
\omega\left( \mathbf{f}, \overline{\mathbf{f}} \right) = \left( A^{T}, B^{T}\right) J \begin{pmatrix} \overline{A} \\ \overline{B} \end{pmatrix} , \quad 
\omega\left( \overline{\mathbf{f}}, \overline{\mathbf{f}} \right) = \left( \overline{A}^{T}, \overline{A}^{T}\right) J \begin{pmatrix} \overline{A} \\ \overline{B}  \end{pmatrix} 
\end{equation}

With this notation, we have that the initial value problem (\ref{eq-EDO-Cartan-2}) becomes: 
\[
\begin{pmatrix}
A' & \overline{A}' \\
B' & \overline{B}'
\end{pmatrix} 
=
\begin{pmatrix}
A & \overline{A} \\
B & \overline{B}
\end{pmatrix}
\begin{pmatrix}
\Sigma & K \\
\mathrm{Id} & \Sigma
\end{pmatrix}
\]
and then
\begin{align}
A' &= A\Sigma + \overline{A} \label{eq-EDO-1} \\ 
B' &= B\Sigma + \overline{B} \label{eq-EDO-2} \\ 
\overline{A}' &= A K + \overline{A}\Sigma \label{eq-EDO-3} \\ 
\overline{B}' &= B K + \overline{B}\Sigma \label{eq-EDO-4} 
\end{align}

Deriving the matrices of (\ref{eq-matrices-omegaff}) and using that $\Sigma^{T}=-\Sigma$, $K^{T}= K$, $\omega\left( \overline{\mathbf{f}} , \mathbf{f}\right)=-\omega\left( \mathbf{f}, \overline{\mathbf{f}} \right)^{T}$ and equations (\ref{eq-EDO-1})--(\ref{eq-EDO-4}), we obtain the system of ODEs  
\[
\left\{
\begin{tabular}{l}
$\frac{d}{d\tau}\omega\left( \mathbf{f}, \mathbf{f} \right) = -\Sigma \omega\left( \mathbf{f}, \mathbf{f} \right) + \omega\left( \mathbf{f}, \mathbf{f} \right)\Sigma + \omega\left( \mathbf{f}, \overline{\mathbf{f}} \right) + \omega\left( \overline{\mathbf{f}}, \mathbf{f} \right)$ \vspace{1mm} \\
$\frac{d}{d\tau}\omega\left( \mathbf{f}, \overline{\mathbf{f}} \right) = -\Sigma \omega\left( \mathbf{f}, \overline{\mathbf{f}} \right) + \omega\left( \mathbf{f}, \overline{\mathbf{f}} \right)\Sigma + \omega\left( \overline{\mathbf{f}}, \overline{\mathbf{f}} \right) + \omega\left( \mathbf{f}, \mathbf{f} \right) K$ \vspace{1mm}  \\
$\frac{d}{d\tau}\omega\left( \overline{\mathbf{f}}, \overline{\mathbf{f}} \right) = K \omega\left( \mathbf{f}, \overline{\mathbf{f}} \right)- \Sigma \omega\left( \overline{\mathbf{f}}, \overline{\mathbf{f}} \right) + \omega\left( \overline{\mathbf{f}}, \mathbf{f} \right) K + \omega\left( \overline{\mathbf{f}}, \overline{\mathbf{f}} \right)\Sigma$ \end{tabular}
\right.
\]
which has the solution: 
\begin{equation}\label{eq-diff-matrices-omegaff}
\omega\left( \mathbf{f}, \mathbf{f} \right) = 0 , \quad 
\omega\left( \mathbf{f}, \overline{\mathbf{f}} \right) = \mathrm{Id} , \quad 
\omega\left( \overline{\mathbf{f}}, \overline{\mathbf{f}} \right) = 0 \, ,
\end{equation}
and by uniqueness of the solution, we have that $F(\tau)=(\mathbf{f}(\tau),\overline{\mathbf{f}}(\tau))$ is a symplectic basis for all $\tau\in I$.
\end{proof}

Now, we can establish the following reconstruction theorem.

\begin{theorem}\label{thm:reconstruction}
Given smooth matrices $\Sigma=\Sigma(\tau)\in\mathbb{R}^{n \times n}$ and $K = K(\tau)\in\mathbb{R}^{n \times n}$, where $\Sigma$ is skew--symmetric and $K$ diagonal for $\tau\in I\subset \mathbb{R}$,  with diagonal elements $k_i$ satisfying the functional relation (\ref{eq-condicion-curvatura}), that is:
$$
\prod_{i=1}^{m} | k_i-  \bar{k} | = 1\, ,
$$
and a symplectic basis $F_0=(\mathbf{f}_0 ,\overline{\mathbf{f}}_0)\subset W$, then there exist a unique Jacobi curve $\Gamma=\Gamma(\tau)$ such that $\Gamma(0)=\mathrm{span}\{\mathbf{f}_0\}\in\mathscr{L}(W)$ and the Cartan matrix related to $\Gamma$ is:
\begin{equation}\label{eq:red_cartan}
\mathbf{C}_{\Gamma}=\begin{pmatrix}
\Sigma & K \\
\mathrm{Id} & \Sigma
\end{pmatrix}  .
\end{equation}
\end{theorem}

\begin{definition}
With the assumptions of Thm. \ref{thm:reconstruction}, we will call the matrix $\mathbf{C}_\Gamma$, given by Eq. (\ref{eq:red_cartan}), the reduced Cartan matrix of $\Gamma$ and it provides the normal form for the Cartan matrix determined by a moving symplectic frame adapted to a Jacobi curve, \textit{cfr.} Sect. \ref{sec:moving_frames}, Eq. (\ref{eq:Cartan_gen}).
\end{definition}

\begin{proof}
Let $F=(\mathbf{f},\overline{\mathbf{f}})$ be the unique solution of the problem (\ref{eq-EDO-Cartan-2}) of lemma \ref{lem-symplectic-f} that can be written as in (\ref{eq-coordinates-ff}). Since $F=F(\tau)$ is a symplectic basis for all $\tau\in I$ then equations (\ref{eq-matrices-omegaff}) become:
\begin{align}
& A^{T}B = B^{T}A \label{eq-sympl-1} \\ 
& A^{T}\overline{B} -B^{T}\overline{A} = \mathrm{Id} \label{eq-sympl-2} \\ 
& \overline{A}^{T}\overline{B} = \overline{B}^{T}\overline{A} \label{eq-sympl-4} 
\end{align}

Let us define the Jacobi curve $\Gamma(\tau)=\mathrm{span}\{\mathbf{f}\}$, so 
\[
\Gamma(\tau) \simeq \begin{bmatrix}
A \\ B
\end{bmatrix} \simeq \begin{bmatrix}
\mathrm{Id} \\ BA^{-1}
\end{bmatrix}
, \qquad
\mathrm{span}\{\overline{\mathbf{f}}\} \simeq \begin{bmatrix}
\overline{A} \\ \overline{B}
\end{bmatrix} \simeq \begin{bmatrix}
\mathrm{Id} \\ \overline{B}~\overline{A}^{-1}
\end{bmatrix}
\]
and so, we denote the symmetric matrices
\[
S=BA^{-1} , \qquad \overline{S}=\overline{B}~\overline{A}^{-1}   .
\]

By equation (\ref{eq-base-simplectica}), we have that 
\[
A^{T}\left( \overline{S} - S \right)\overline{A}=\mathrm{Id} \Longrightarrow  \overline{A}=\left( \overline{S} - S \right)^{-1}(A^{T})^{-1}
\]
and therefore
\begin{equation}\label{eq-MMbarra-simetrica}
\left(A^{-1}\overline{A}\right)^{T} = \left(A^{-1}\left( \overline{S} - S \right)^{-1}(A^{T})^{-1}\right)^{T} = A^{-1}\left( \overline{S} - S \right)^{-1}(A^{T})^{-1} = A^{-1}\overline{A}
\end{equation}
so $A^{-1}\overline{A}$ is a symmetric matrix.

Now, let us compute the blocks of the Cartan matrix (\ref{eq-Cartan-tau}) for the curve $\Gamma$, where we denote
\[
\mathbf{C}_{\Gamma} = \begin{pmatrix}
\mathbf{C}_{11} & \mathbf{C}_{12} \\
\mathbf{C}_{21} & \mathbf{C}_{11}
\end{pmatrix}    .
\] 

First, we have
\begin{align*}
S' &= B'A^{-1}-BA^{-1}A'A^{-1} =  & \text{ (because of (\ref{eq-EDO-1}) and (\ref{eq-EDO-2}))} \\
&= \left(B\Sigma + \overline{B}\right)A^{-1}-BA^{-1}\left(A\Sigma + \overline{A}\right)A^{-1} = & \\
&= \left(\overline{B}-BA^{-1}\overline{A}\right)A^{-1} = & \text{(because of (\ref{eq-sympl-2}))} \\
&= \left(\left(A^{T}\right)^{-1} + \left(A^{T}\right)^{-1}B^{T}\overline{A} -BA^{-1}\overline{A}\right)A^{-1} = & \\
&= \left(\left(A^{T}\right)^{-1} + \left[\left(BA^{-1}\right)^{T} -BA^{-1}\right]\overline{A}\right)A^{-1} = & \text{ ( since } S=BA^{-1} \text{  is symmetric )} \\
&= \left(AA^{T}\right)^{-1} \, ,
\end{align*}
then,
\begin{equation}\label{eq-A-normalized}
 A^{T}S'A= \mathrm{Id}  \, ,
\end{equation}
corresponding to the block $\mathbf{C}_{21}$.

Since $S'=\left(AA^{T}\right)^{-1}$, because of equation (\ref{eq-EDO-1}), we obtain:
\[
S'' = -S'\left( \overline{A}A^{T} + A \overline{A}^{T} \right) S' \, ,
\]
and
\[
S''' = 2S'\left( \overline{A}A^{T} + A \overline{A}^{T} \right) S'\left( \overline{A}A^{T} + A \overline{A}^{T} \right) S'-2S'\left( A\Delta A^{T} + \overline{A} ~\overline{A}^{T} \right) S' \, .
\]
Then,
\[
\mathbb{S}(S) = \frac{1}{2}\left( \overline{A}A^{T} + A \overline{A}^{T} \right) S'\left( \overline{A}A^{T} + A \overline{A}^{T} \right) S'-2\left( A K A^{T} + \overline{A} ~\overline{A}^{T} \right) S' \, .
\]
The block $\mathbf{C}_{11}$ of $\mathbf{C}_{\Gamma}$ is  
\[
\mathbf{C}_{11}=\frac{1}{2}\left[K_\tau^{-1} A'_{\tau} - \left(K_\tau^{-1} A'_{\tau}\right)^{T} \right]
\]
so using (\ref{eq-EDO-1}) and the symmetry of $A^{-1} \overline{A}$ we have:
\[
\mathbf{C}_{11}=\Sigma +\frac{1}{2}\left[A^{-1} \overline{A} - \left(A^{-1} \overline{A}\right)^{T} \right] = \Sigma  \, .
\]

Finally, the block $\mathbf{C}_{12}$ is
\begin{equation}\label{eq-block-C12}
\mathbf{C}_{12}=-\frac{1}{2}A^{-1}\mathbb{S}(S)A
\end{equation}
so using (\ref{eq-MMbarra-simetrica}), we get:
\begin{align}
\mathbf{C}_{12} & = -\frac{1}{4}\left( A^{-1}\overline{A}A^{T} +  \overline{A}^{T} \right) S'\left( \overline{A} + A \overline{A}^{T} (A^{T})^{-1}\right)  + \left( K  + A^{-1}\overline{A} ~\overline{A}^{T} (A^{T})^{-1}\right) = \nonumber \\
&= -\frac{1}{4}\left( A^{-1}\overline{A}A^{-1} +  \overline{A}^{T}(A^{T})^{-1}A^{-1} \right) \left( \overline{A} + A \overline{A}^{T} (A^{T})^{-1}\right)  + \left( K  + A^{-1}\overline{A} ~\overline{A}^{T} (A^{T})^{-1}\right) = \nonumber  \\
&= -\frac{1}{4}\left( A^{-1}\overline{A}A^{-1}\overline{A} +  \left(A^{-1}\overline{A}\right)^{T}A^{-1} \overline{A} + A^{-1}\overline{A}\left(A^{-1}\overline{A}\right)^{T} +  \left(A^{-1}\overline{A}A^{-1}\overline{A}\right)^{T} \right) + \nonumber  \\
& + \left( K + A^{-1}\overline{A} \left(A^{-1}\overline{A}\right)^{T}\right) =  - \left(A^{-1}\overline{A}\right)^{2} + K  +  \left(A^{-1}\overline{A}\right)^{2} = K \label{eq-C12-Delta}
\end{align}
concluding that the Cartan matrix is:
\[
\mathbf{C}_{\Gamma} = \begin{pmatrix}
\Sigma & K \\
\mathrm{Id} & \Sigma
\end{pmatrix} \, ,
\]
as claimed.
\end{proof}

\begin{remark}
Observe that equation (\ref{eq-C12-Delta}) implies, because of (\ref{eq-block-C12}), that the matrix $A\in\mathbb{R}^{n \times n}$ in the proof corresponds to the matrix of eigenvectors of $\mathbb{S}(S)$ whose matrix of eigenvalues is $-\frac{1}{2}K \in\mathbb{R}^{n \times n}$. 
Since $S'$ is symmetric and it can be written by the product $S'=(A^{-1})^{T}A^{-1}$ then $S'$ is positive definite. 
Moreover, by equation (\ref{eq-C12-Delta}), the eigenvectors of $\mathbb{S}(S)$ given by the matrix $A$ are normalized by the scalar product given by $S'$. 
\end{remark}

It is obvious that the parameters $k_i$, diagonal elements of the matrix $K$, and the entries $\sigma_{ij}$, $1 \leq i < j \leq n$, of the matrix $\Sigma$, are absolute curvatures for the Jacobi curve $\Gamma$.  

Note that in the reconstruction theorem above, Thm. \ref{thm:reconstruction}, the choice of the adapted symplectic basis is not unique in the sense, that any other basis $(\epsilon_1 f_1, \ldots, \epsilon_n f_n, \epsilon_1 \bar{f}_1, \ldots, \epsilon_n \bar{f}_n)$, with $\epsilon_k = \pm 1$, with give rise to the same Jacobi curve. The same will happen if we reorder the vectors $f_k, \bar{f_k}$, or, if there are multiple eigenvalues, the structure of the eigenvectors $f_k$ can be more general.   Thus, in the generic situation when the eigenvalues $k_1, \ldots, k_n$ are all different and are ordered in ascending order, that is, $k_1 < \cdots < k_n$, we will say that two reduced Cartan matrices:
$$
C = \begin{pmatrix}
\Sigma & K \\
\mathrm{Id} & \Sigma
\end{pmatrix} \, , \quad \mathrm{and\,\,} \overline{C} = \begin{pmatrix}
\overline{\Sigma} & \overline{K} \\
\mathrm{Id} & \overline{\Sigma}
\end{pmatrix} \, ,
$$
are equivalent if there exists a diagonal matrix $P$ with diagonal entries $\pm 1$, such that $\overline{K} = P K P$, and $\overline{\Sigma} = P \Sigma P$. 


\subsection{The structure of Jacobi curves in $\mathrm{dim}(W)= 4$}\label{sec:dim4}

Whenever $n=2$, condition (\ref{eq-condicion-curvatura}) becomes:
\[
\frac{1}{4}\left(k_1-k_2\right)^2 = 1 \quad \Leftrightarrow \quad \vert k_1-k_2 \vert = 2
\]
and there exists $\mu=\mu(\tau)$ such that:
\[
\left\{
\begin{matrix}
k_1(\tau)=\mu(\tau)-1 \\
k_2(\tau)=\mu(\tau)+1
\end{matrix}
\right.
\]
So, without any lack of generality, we can write the reduced Cartan matrix as:
\[
\mathbf{C}_{\tau}= 
\begin{pmatrix}
0 & \sigma & -\left(\frac{\mu-1}{2}\right) & 0 \\
-\sigma & 0 & 0 & -\left(\frac{\mu+1}{2}\right) \\
1 & 0 & 0 & \sigma \\
0 & 1 & -\sigma & 0 
\end{pmatrix} 
\]

From the functions $\mu=\mu(\tau)$ and $\sigma=\sigma(\tau)$ we can compute the matrix $M_{\tau}$ solving the homogeneous linear system of ODEs given by:
\begin{equation}\label{eq-EDO-Cartan}
\left\{ \begin{tabular}{l}
$\mathbf{f}'_1 = -\sigma \mathbf{f}_2 + \overline{\mathbf{f}}_1$ \\
$\mathbf{f}'_2 = \sigma \mathbf{f}_1 + \overline{\mathbf{f}}_2$ \\
$\overline{\mathbf{f}}'_1 = -\left(\frac{\mu-1}{2}\right) \mathbf{f}_1 -\sigma \overline{\mathbf{f}}_2$ \\
$\overline{\mathbf{f}}'_2 = -\left(\frac{\mu+1}{2}\right) \mathbf{f}_2 +\sigma \overline{\mathbf{f}}_1$
\end{tabular} \right.
\end{equation}

Observe that, provided that $\sigma\equiv 0$, the system (\ref{eq-EDO-Cartan}) can be decoupled as:
\begin{equation}\label{eq-EDO-Cartan-1}
\left\{ \begin{tabular}{l}
$\mathbf{f}'_1 = \overline{\mathbf{f}}_1$ \\
$\overline{\mathbf{f}}'_1 = -\left(\frac{\mu-1}{2}\right) \mathbf{f}_1 $ 
\end{tabular} \right.
~, \qquad 
\left\{ \begin{tabular}{l}
$\mathbf{f}'_2 =  \overline{\mathbf{f}}_2$ \\
$\overline{\mathbf{f}}'_2 = -\left(\frac{\mu+1}{2}\right) \mathbf{f}_2 $
\end{tabular} \right.
\end{equation}

\subsubsection{Case $\sigma\equiv0$ and $\mu\equiv \pm 1$.}

The solutions of an initial value problem of the form: 
\[
\left\{ \begin{tabular}{l}
$\mathbf{f}' =  \overline{\mathbf{f}}$ \\
$\overline{\mathbf{f}}' = \eta \mathbf{f} $ \\
$\mathbf{f}(0)=\mathbf{v}$ and $\overline{\mathbf{f}}(0)=\overline{\mathbf{v}}$
\end{tabular} \right.
\]
with constant $\eta$ depends on the sign of $\eta$. So, if $\eta>0$, then we get:
\[
\left\{
\begin{tabular}{l}
$\mathbf{f}(\tau)=\cosh(\sqrt{\eta}\tau) \cdot \mathbf{v} +  \frac{1}{\sqrt{\eta}}\sinh(\sqrt{\eta}\tau)\cdot \overline{\mathbf{v}} $  \\
$\overline{\mathbf{f}} (\tau)=\sqrt{\eta} \sinh(\sqrt{\eta}\tau)\cdot \mathbf{v} +  \cosh(\sqrt{\eta}\tau)\cdot \overline{\mathbf{v}} $
\end{tabular}
\right.
\]
If $\eta<0$, then the solution is given by:
\[
\left\{
\begin{tabular}{l}
$\mathbf{f}(\tau)= \cos(\sqrt{\vert \eta\vert}\tau) \cdot \mathbf{v} +  \frac{1}{\sqrt{\vert \eta\vert }}\sin(\sqrt{\vert \eta\vert }\tau) \cdot \overline{\mathbf{v}}$  \\
$\overline{\mathbf{f}} (\tau)=-\sqrt{\vert \eta\vert }\sin(\sqrt{\vert \eta\vert }\tau) \cdot \mathbf{v} +  \cos(\sqrt{\vert \eta\vert }\tau)\cdot \overline{\mathbf{v}} $
\end{tabular}
\right.
\]
and if $\eta=0$:
\[
\left\{
\begin{tabular}{l}
$\mathbf{f}(\tau)=\mathbf{v} + \overline{\mathbf{v}}\cdot \tau $  \\
$\overline{\mathbf{f}} (\tau)=\overline{\mathbf{v}} $
\end{tabular}
\right.
\]

Observe that if $\mu=\pm 1$, one of the eigenvalues $k_i\equiv 0$ and therefore the vector $\overline{\mathbf{f}}_i$ is constant. On the other hand, if $\overline{\mathbf{f}}_i (\tau)$ is constant along the Jacobi curve, in virtue of the fourth equation of the system (\ref{eq-EDO-Cartan}), since $\mathbf{f}_i$ and $\overline{\mathbf{f}}_j$ are linearly independent, then $\sigma\equiv 0$ and $k_i\equiv 0$. 
That is, for $i=1,2$ we have:
\[
\boxed{ \overline{\mathbf{f}}_i (\tau) \text{ constant}\Longleftrightarrow \left\{ \begin{matrix} \sigma\equiv 0  \\ k_i\equiv 0 \end{matrix} \right.  }
\]

\begin{example}
In this example, we will assume that $\sigma=0$ and $\mu=-1$, so $k_1=-2$ and $k_2=0$ and then the solutions $\mathbf{f}_1$ and $\mathbf{f}_2$ of (\ref{eq-EDO-Cartan}) are given by:
\[
\left\{
\begin{tabular}{l}
$\mathbf{f}_1(\tau)=\cosh(\tau) \cdot \mathbf{v}_1 +  \sinh(\tau)\cdot \overline{\mathbf{v}}_1 $\\
$\mathbf{f}_2(\tau)=\mathbf{v}_2 + \overline{\mathbf{v}}_2\cdot \tau $  
\end{tabular}
\right.
\]
If we consider the initial symplectic basis:
\[
\left\{
\begin{tabular}{l}
$\mathbf{v}_1 =\mathbf{e}_1 $  \\
$\mathbf{v}_2 = \mathbf{e}_2$ \\
$\overline{\mathbf{v}}_1= \mathbf{e}_1 + \overline{\mathbf{e}}_1$ \\
$\overline{\mathbf{v}}_2= \mathbf{e}_2 + \overline{\mathbf{e}}_2$ 
\end{tabular}
\right.
\]
with respect to the basis $\left(\mathbf{e}_1 ,\mathbf{e}_2 ,\overline{\mathbf{e}}_1, \overline{\mathbf{e}}_2\right)$, then  
\begin{equation}\label{eq-f-ejemplo}
\left\{
\begin{tabular}{l}
$\mathbf{f}_1(\tau)=\left(\cosh(\tau)  +  \sinh(\tau)\right) \mathbf{e}_1 + \sinh(\tau)\overline{\mathbf{e}}_1$  \\
$\mathbf{f}_2(\tau)= (1+\tau)\mathbf{e}_2 + \tau\overline{\mathbf{e}}_2$
\end{tabular}
\right.
\end{equation}
and the matrix $M_{\tau}$ can be written by the corresponding coordinates with respect to  $(\mathbf{e}_1,\mathbf{e}_2)$. That is, 
\[
M_{\tau} = \begin{pmatrix}
\cosh(\tau)  +  \sinh(\tau) & 0 \\
0 & 1+\tau
\end{pmatrix}
\] 
and so, by (\ref{eq-SprimaMM}), we obtain 
\[
S'_{\tau}=\left(M_{\tau}M_{\tau}^{T}\right)^{-1} = \begin{pmatrix}
\left(\cosh(\tau)  +  \sinh(\tau)\right)^{-2} & 0 \\
0 & (1+\tau)^{-2}
\end{pmatrix}
\]

Also observe that it is possible to compute \textbf{directly} the matrix $S_{\tau}$ from the coordinates of $(\mathbf{f}_1,\mathbf{f}_2)$ with respect to $(\overline{\mathbf{e}}_1,\overline{\mathbf{e}}_2)$ by using the expression of the first equation of the system (\ref{eq-f-fbarra}). 
By (\ref{eq-f-ejemplo}), we have
\[
S_{\tau}M_{\tau}=\begin{pmatrix}
\sinh(\tau) & 0 \\
0 & \tau
\end{pmatrix}
\]
and then
\[
S_{\tau}=\begin{pmatrix}
\sinh(\tau) & 0 \\
0 & \tau
\end{pmatrix}
\begin{pmatrix}
 \frac{1}{\cosh(\tau)  +  \sinh(\tau)} & 0 \\
0 & \frac{1}{1+\tau}
\end{pmatrix} = 
\begin{pmatrix}
\frac{\sinh(\tau)}{\cosh(\tau)  +  \sinh(\tau)} & 0 \\
0 & \frac{\tau}{1+\tau} 
\end{pmatrix} 
\]
\end{example} 

\begin{example}
Now, we will assume that $\sigma=0$ and $\mu=1$ (this means $k_1=0$ and $k_2=2$) and the solutions of (\ref{eq-EDO-Cartan}) are:
\[
\left\{
\begin{tabular}{l}
$\mathbf{f}_1(\tau)=\mathbf{v}_1 + \overline{\mathbf{v}}_1\cdot \tau $ \\
$\mathbf{f}_2(\tau)=\cos(\tau) \cdot \mathbf{v}_2 +  \sin(\tau)\cdot \overline{\mathbf{v}}_2 $  
\end{tabular}
\right.
\]
Again, if the initial vectors are:
\[
\left\{
\begin{tabular}{l}
$\mathbf{v}_1 =\mathbf{e}_1 $  \\
$\mathbf{v}_2 = \mathbf{e}_2$ \\
$\overline{\mathbf{v}}_1= \mathbf{e}_1 + \overline{\mathbf{e}}_1$ \\
$\overline{\mathbf{v}}_2= \mathbf{e}_2 + \overline{\mathbf{e}}_2$ 
\end{tabular}
\right.
\]
then  
\begin{equation}\label{eq-f-ejemplo2}
\left\{
\begin{tabular}{l}
$\mathbf{f}_1(\tau)= (1+\tau)\mathbf{e}_1 + \tau\overline{\mathbf{e}}_1$ \\
$\mathbf{f}_2(\tau)= \left(\cos(\tau)  +  \sin(\tau)\right) \mathbf{e}_2 + \sin(\tau)\overline{\mathbf{e}}_2$
\end{tabular}
\right.
\end{equation}
and the matrix $M_{\tau}$ is written as:
\[
M_{\tau} = \begin{pmatrix}
1+\tau & 0 \\
0 & \cos(\tau)  +  \sin(\tau)
\end{pmatrix} \, ,
\] 
and we obtain 
\[
S'_{\tau}=\left(M_{\tau}M_{\tau}^{T}\right)^{-1} = \begin{pmatrix}
(1+\tau)^{-2} & 0 \\
0 & \left(1+\cos 2\tau\right)^{-1}
\end{pmatrix} \, ,
\]
or directly, we have
\[
S_{\tau}=
\begin{pmatrix}
 \frac{\tau}{1+\tau} & 0 \\
0 & \frac{\sin(\tau)}{\cos(\tau)  +  \sin(\tau)} 
\end{pmatrix}  \, .
\]
\end{example} 


\subsection{The reconstruction theorem and the algorithmic classification of admissible Jacobi curves}\label{sec:algorithm}

The theory and classification of Jacobi curves with respect to the action of the conformal symplectic group developed so far allows for an algorithmic description.   In fact, such algorithmic classification can be implemented in any modern symbolic manipulation environment like Mathematica or Maple.    We will succinctly sketch it in  what follows.

Let $\Gamma = \Gamma (t) = \left[\begin{array}{c} I \\ S_t \end{array} \right]_{(\mathbf{e}, \overline{\mathbf{e}})}$ be a Jacobi curve in $\mathscr{L}(W)$, with $W$ a conformal symplectic space of dimension $2n$ with conformal symplectic class $[\omega]$.\\

A)  First, we will check its admissibility, \textit{cfr.} Def. \ref{def:Ricci}.

\begin{enumerate}
\item  Compute $S_t'$, which is the matrix of a bilinear form in $\Gamma (t)$ with respect to the basis $\mathbf{e} + \overline{\mathbf{e}} S_t$, and we check that is definite positive, \textit{cfr.} Remark \ref{rem:Sprime}.   If it were positive negative we change $\omega$ by $-\omega$.

\item  Compute the matrix $\mathbb{S}(S_t)$, \textit{cfr.} Eq. (\ref{eq:Ricci_Schwarz}), which is the matrix of the Ricci curvature tensor $R_\Gamma (t) \colon \Gamma (t) \to \Gamma (t)$ in the basis $\boldsymbol{\epsilon} = \mathbf{e} + \overline{\mathbf{e}} S_t$, and we check that it is diagonalisable, \textit{cfr.} Thm. \ref{thm:symmetric}.  We check that all its eigenvalues $\mu_i$, $i = 1, \ldots, n$, are different and we reorder them, if necessary, such that $\mu_1 < \mu_2 < \cdots < \mu_n$.

\item Compute $\zeta (t)$ using Eq. (\ref{eq-zeta}), and obtain the geometric arc parameter $d \tau = \zeta (t) dt$. 

\item  Check that $\det \left(  \mathbb{S}(S_t - \frac{1}{n} \mathrm{Tr\,} (\mathbb{S}(S_t))\right) \neq 0$, which amounts to $(\bar{\mu} - \mu_1) \cdots (\bar{\mu} - \mu_n) ) = \zeta^{2n}$, where $\bar{\mu} = \frac{1}{n} (\mu_1 + \cdots + \mu_n)$ and $\zeta (t)$ is given by Eq. (\ref{eq-zeta}).
\end{enumerate}

B) We compute the absolute curvature operator $\mathcal{R}_\Gamma (t)$, \textit{cfr.} Def. (\ref{def:absolute_curvature}), Eq. (\ref{eq:absolute_curvature}).\\

C) Next, compute the eigenvalues $k_i$, $i = 1, \ldots, n$, of $R_\Gamma (t)$:
$$
k_i = \frac{1}{\zeta^2} (\mu_i - \mathbb{S}(\zeta) )\, , \qquad k_1 < \cdots < k_n \, .
$$

Let $S_t^0 = S_t - 2 S_t' \left(S_t'' - \frac{\zeta'}{\zeta} S_t' \right)^{-1} S_t'$ be the coordinates of the derivative curve depending on $t$, see Rem. \ref{rem:derivative_nongeom}. 

D) Construct $\mathbf{f} = (f_1, \ldots, f_n) = \boldsymbol{\epsilon} M$, the $S_t'$-orthonormal basis of eigenvectors of $R_\Gamma (t)$, and $\overline{\mathbf{f}}= \boldsymbol{\epsilon} \overline{M}$, with $\overline{M} = \mathbf{f} (S_t - S_t^0) (M^T)^{-1}$, \textit{cfr.} Lemma \ref{lema-base-simplectica}, Eq. (\ref{eq-base-simplectica}). 

The basis $F = (\mathbf{f}, \overline{\mathbf{f}})$ thus constructed is the symplectic Frenet basis adapted to the curve $\Gamma(t)$, \textit{cfr.} Sect. \ref{sec:moving_frames}.

E) The equation $F' = F C$ provides the reduced Cartan matrix $C$.   Namely, $\mathbf{f}' = \mathbf{f} \Sigma + \overline{\mathbf{f}}$ will provide the matrices:
$$
K = \left(\begin{array}{ccc} k_1 &   &   \\   & \ddots &   \\ & & k_n \end{array} \right) \, , \qquad \Sigma = \frac{1}{2\zeta(t)} \left(  \widetilde{M}^{-1}_t \widetilde{M}'_{t} - \left(\widetilde{M}^{-1}_t \widetilde{M}'_{t}\right)^{T}  \right) \, 
$$
\textit{cfr.} (\ref{eq-matriz-C-2}), with $k_1 < \cdots < k_n$ for all $t$, that ends the computation.\\

In this spirit and  following the comments after the proof of the reconstruction theorem, Thm. \ref{thm:reconstruction}, we will end this section by restating it as follows:

\begin{theorem} (Classification theorem for Jacobi curves). 
\begin{enumerate}
\item Two Jacobi curves $\Gamma = \Gamma (t)$ and $\overline{\Gamma} = \overline{\Gamma} (t)$ parametrised by a geometric arc are $CSp$-equivalent if and only if they  have equivalent reduced Cartan matrices.

\item   Two Jacobi curves $\Gamma = \Gamma (t)$ and $\overline{\Gamma} = \overline{\Gamma} (t)$ are $CSp$-equivalent if and only if $ds_\Gamma (t) = ds_{\overline{\Gamma}}(t)$ and their reduced Cartan matrices are equivalent.

\item Fixed a reduced Cartan matrix $C(t)$:
$$
C (t) = \begin{pmatrix}
\Sigma_t & K_t \\
\mathrm{Id} & \Sigma_t
\end{pmatrix} 
$$
 $t \in I \subset \mathbb{R}$, depending smoothly on $t$, with $k_1 < \cdots < k_n$, and $\prod |k_i - \bar{k}| = 1$, and a 1-form $ds (t) \neq 0$ for all $t$, there exists a Jacobi curve $\Gamma (t)$ such that $C_\Gamma (t) = C(t)$ and $ds_\Gamma (t) = ds(t)$.  Moreover this curve is unique up to transformations by the conformal symplectic group $CSp$.
\end{enumerate}
\end{theorem}

In the particular instance of $\dim W = 4$, \textit{cfr.} Sect. \ref{sec:dim4}, given a Jacobi curve $\Gamma (t)$, we get $k_1(t) = - \frac{\mu(t)-1}{2}$, $k_2(t) = - \frac{ \mu(t) + 1}{2}$, and $\sigma (t)$ is obtained from $f_1' = - \sigma f_2 + f_1$ (if $\sigma < 0$, we take $-f_2$ instead), and we get the curvatures $\sigma_\Gamma (t) = \sigma (t) \geq 0$, and $\mu_\Gamma (t) = \mu(t)$ for the curve $\Gamma(t)$.   Conversely, given the functions $\mu (t) \neq 0$, fand $\sigma (t) \geq 0$, or all $t$, and a 1-form $ds = \zeta (t) dt$, $\zeta (t) > 0$, for all $t$, there exists a Jacobi curve $\Gamma (t)$ in a 4-dimensional conformal symplectic space such that $\mu_\gamma = \mu$, $\sigma_\Gamma = \sigma$, and $ds_\Gamma = ds$, and all of them are of the form $\overline{\Gamma} = \phi (\Gamma)$, with $\phi \in CSp$.


\section{Cycles in $ \mathscr{L}\left(W\right) $}\label{sec:cycles}

We will end this work by observing that Jacobi curves with null Ricci curvature are cycles in the Lagrangian Grassmannian, that is, the closure at infinite of affine lines.
Given a Lagrangian subspace $\overline{\Lambda}\in \mathscr{L}\left(  W\right) $, there are two classes of affine lines $L$ in the affine space  $\overline{\Lambda}^{\pitchfork}$:

\begin{description}
\item[Regular] Affine lines $L$ verifying that $\forall\Lambda_{1},\Lambda_{2}\in L$,
$\Lambda_{1}\neq\Lambda_{2}$, then $\Lambda_{1}\cap\Lambda_{2}=0$.

\item[Singular] Affine lines $L$ verifying that $\forall\Lambda_{1},\Lambda_{2}\in L$,
$\Lambda_{1}\neq\Lambda_{2}$, then $\Lambda_{1}\cap\Lambda_{2}\neq0$.
\end{description}
To show that this is true it suffices to notice that if $L$ is a line in
$\overline{\Lambda}^{\pitchfork}$ passing through $\Lambda\in\overline{\Lambda
}^{\pitchfork}$, then taking $S=\mathcal{S}_{\overline{\Lambda}^{\pitchfork}%
}^{\Lambda}$, then $L=L\left(  t\right)  =t\overset{.}{S}_{0}$ provided that
$\det\overset{.}{S}_{0}\neq0$.  In such case $L$ is regular, because $L\left(
a\right)  \cap L\left(  b\right)  =0$ if $a\neq b$, that is, $b\overset{.}{S}%
_{0}-a\overset{.}{S}_{0}=\left(  b-a\right)  \overset{.}{S}_{0}$  and
$\det\left(  b-a\right)  \overset{.}{S}_{0}\neq0$.  If $\det\overset
{.}{S}_{0}=0$, the line  $L$ is obviously singular.

\begin{remark}
A regular line $L$ in $\overline{\Lambda}^{\pitchfork}$ is
an admissible Jacobi curve in $ \mathscr{L}\left(  W\right)  $.
\end{remark}

\begin{definition}[Cycle]  Let $\overline{\Lambda}\in \mathscr{L}\left(  W\right)$ be a Lagrangian subspace and  $L$ an
affine line in $\overline{\Lambda}^{\pitchfork}$.  We will call a cycle the set
$C=\widetilde{L}=L\cup\left\{  \overline{\Lambda}\right\}$.  The cycle $C$ will be said to be regular if  $L$ is regular, otherwise singular.
\end{definition}

In what follows we will show that regular cycles are characterised as maximal Jacobi curves with vanishing curvature with respect to any projective parametrisation.

\begin{proposition}\label{C2} 
If $C$ is a regular cycle then for each $\Lambda\in C$ we have that
$C\backslash\left\{  \Lambda\right\} $ is a Jacobi curve  and its projective reparametrizations $\Gamma = \Gamma(t)$ have vanishing curvature operator, $\mathcal{R}_{\Gamma} =0$.
\end{proposition}

\begin{proof}
Indeed, if $C=\widetilde{L}=L\cup\left\{  \overline{\Lambda}\right\}$
with  $L$ a affine line in $\overline{\Lambda}^{\pitchfork}$,  and
$\Lambda\in C$ with $\Lambda\neq\overline{\Lambda}$, then taking coordinates $S=\mathcal{S}_{\overline{\Lambda}^{\pitchfork}}^{\Lambda}$,
we can write $L=L\left(  t\right)$ with $S\left(  L\left(  t\right)
\right)  =t\overset{.}{S}_{0}$, and $\det\overset{.}{S}_{0}\neq0$, which is an affine parametrization of  $L=C\backslash\left\{  \overline{\Lambda
}\right\}$. This parametrization is projective because $\mathbb{S}\left(
t\overset{.}{S}_{0}\right)  =0$ and, consequently$Ric_{\Gamma}\left(  t\right) =0.$ 

Taking now coordinates $S^{-1}=\mathcal{S}_{\Lambda^{\pitchfork}%
}^{\overline{\Lambda}}$ in $\overline{\Lambda}^{\pitchfork}\cap\Lambda
^{\pitchfork}$, it is easy to see that $S^{-1}\left(  L\left(  t\right)  \right)  =\left(
1/t\right)  \overset{.}{S}_{0}^{-1}$($t\neq0$) is a parametrization of 
$C\backslash\left\{  \Lambda,\overline{\Lambda}\right\}  = \overline
{L}\backslash\left\{  \overline{\Lambda}\right\} $, where  $\overline
{L}=C\backslash\left\{  \Lambda\right\} $ is an affine line in   $\Lambda^{\pitchfork}$.

Last, observe that any projective reparametrization 
$\overline{L}=C\backslash\left\{  \Lambda\right\}  $ has the form:
\[
S_{t}=\frac{at+b}{ct+d} \, S_{0}%
\]
and the Schwarzian derivative  $\mathbb{S}\left(  S_{t}\right)  =0$. Then $\mathcal{R}_\Gamma = 0$ because of (\ref{eq:Ricci_Schwarz}).
\end{proof}

\begin{corollary} If $C$ is a regular cycle and $\Lambda\in C$, then $C\backslash\left\{
\Lambda\right\}$ is an affine line in $\overline{\Lambda}$.
\end{corollary}

\begin{remark}
If $C$ is a regular  cycle like in  Prop. \ref{C2}, continuing with the notation in the proof, 
we get that the map $\varphi:\mathbb{RP}^{1}=\mathbb{R\cup}\left\{
\infty\right\}  \rightarrow C$ given by $t\mapsto\Gamma\left(  t\right) $,
$\infty\mapsto\overline{\Lambda}$ is a bijection that preserves the double ratio. 
It is because of this that we say that  $\varphi:\mathbb{RP}%
^{1}\rightarrow C$ is a GPP (global projective parametrization).
The GPP's of the cycle  $C$ are determined up to homographies: $h:\mathbb{RP}^{1}\rightarrow\mathbb{RP}^{1}$. 

Moreover, given the GPP  $\varphi:\mathbb{RP}^{1}\rightarrow C$, for each  $t_{0}=\left(  t_{0}:1\right)  \in\mathbb{RP}^{1}$, the map:
\[
\varphi:\mathbb{RP}^{1}\backslash\left\{  t_{0}\right\}  \rightarrow
C\backslash\varphi\left(  t_{0}\right)
\]
is an affine isomorphism, where  $\mathbb{RP}^{1}\backslash\left\{
t_{0}\right\} $ has the canonical affine structure given by the fact that the map:
\[
\mathbb{R\rightarrow RP}^{1}\backslash\left\{  t_{0}\right\}  ,\text{
}t\mapsto\frac{1}{t-t_{0}}=\left(  1:t-t_{0}\right) \, ,
\]
is an affine isomorphism.
\end{remark}

\begin{proposition}
If $\Gamma=\Gamma\left(  t\right)$ is a Jacobi curve with projective parameter and such that $\mathcal{R}_{\Gamma}\left(  t\right)  =0$, $\forall t$, 
then the image of $\Gamma$ is contained in a regular cycle.
\end{proposition}

\begin{proof}   The proof is based on the fact that the map $t\mapsto S_{t}$ that satisfies $\mathbb{S}\left(  S_{t}\right)  =0$
is uniquely determined by its values $S_{0}$, $S_{0}^{\prime}$,
$S_{0}^{\prime\prime}$ and, conversely,
 given $S_{0},$ $\overset{.}%
{S}_{0},$ $\overset{..}{S}_{0}$, there exists a unique solution  $t\mapsto
S_{t}$ of the equation $\mathbb{S}\left(  S_{t}\right)  =0$ with $S_{0}=S_{0}$, $S_{0}^{\prime}=\overset{.}{S}_{0}$, $S_{0}^{\prime\prime}=\overset{..}{S}_{0}$. Thus the only solutions of  $\mathbb{S}\left(S_{t}\right)  =0$ have the form:
\[
S_{t}=\frac{at+b}{ct+d}\overset{.}{S}_{0}%
\]
because given $S_{0},$ $\overset{.}{S}_{0},$ $\overset{..}{S}_{0}$, the equations
$S_{0}=S_{0}$, $S_{0}^{\prime}=\overset{.}{S}_{0}$, $S_{0}%
^{\prime\prime}=\overset{..}{S}_{0}$ determine the matrix $\left(
\begin{array}
[c]{cc}%
a & b\\
c & d
\end{array}
\right)  $ up to multiplicative constants.
\end{proof}

\begin{remark}\label{rem}
Given two different points  $\Lambda_{1},\Lambda_{2}\in\overline
{\Lambda}^{\pitchfork}$, with $\overline{\Lambda}\in \mathscr{L}\left(  W\right)
$ there is at least one cycle  $C=L\cup\left\{  \overline{\Lambda}\right\} $ passing through them, where
$L=\left\langle \Lambda_{1},\Lambda_{2}\right\rangle _{\overline{\Lambda
}^{\pitchfork}}$ is the affine line generated by $\Lambda_{1},\Lambda_{2}$ in the affine space
$\overline{\Lambda}^{\pitchfork}$.  We may ask if there are other cycles containing both points.

Taking any $\Lambda
_{0}\in\overline{\Lambda}^{\pitchfork}$  and affine coordinates
$S=\mathcal{S}_{\overline{\Lambda}^{\pitchfork}}^{\Lambda}$ in $\overline
{\Lambda}^{\pitchfork}$ in such a way that the equation of $L$ will be $S_{t}=tS_{1}$
and $S\left(  \Lambda_{0}\right)  =S^{0}$, we then take, $\widetilde{S}=\left(
S-S^{0}\right)  ^{-1}=\mathcal{S}_{\overline{\Lambda}_{0}^{\pitchfork}%
}^{\overline{\Lambda}}$ en $\overline{\Lambda}^{\pitchfork}\cap\overline
{\Lambda}_{0}^{\pitchfork}$, and the equation for $L$ in coordinates
$\widetilde{S}$ becomes $\widetilde{S}_{t}=\left(  tS_{1}-S^{0}\right)  ^{-1}$, whose
image is not an affine line in general in $\overline{\Lambda}_{0}^{\pitchfork}$ unless $\Lambda_{0}\in L$ (that is, $S_{0}=t_{0}S^{0}$. 
This implies that $\widetilde{C}=\left\langle \Lambda_{1},\Lambda
_{2}\right\rangle _{\overline{\Lambda}_{0}^{\pitchfork}}\cup\left\{
\overline{\Lambda}_{0}^{\pitchfork}\right\} $ is another cycle containing
$\Lambda_{1},\Lambda_{2}$
\end{remark}

\begin{definition}
Three different points $\Lambda_{1},\Lambda_{2},\Lambda_{3}\in \mathscr{L}\left(  W\right)$
are called concyclic  if there exists $\Lambda_{0}\in \mathscr{L}\left(  W\right)  $, such that $\Lambda_{1},\Lambda_{2}%
,\Lambda_{3}$ are colinear in the affine space $\Lambda
_{0}^{\pitchfork}$. The three points  $\Lambda_{1},\Lambda_{2},\Lambda_{3}$ are said to be in general position
(GP for short) if $\Lambda_{i}\cap\Lambda_{j}=0$ for $i\neq j$.
\end{definition}

If $\Lambda_{1},\Lambda_{2},\Lambda_{3}\in \mathscr{L}\left(  W\right)$
are GP concyclic, we denote by $C\left(  \Lambda_{1},\Lambda_{2},\Lambda_{3}\right)$ the set:
$$
C\left(  \Lambda_{1},\Lambda_{2},\Lambda_{3}\right) = \left\{  \Lambda
_{1},\Lambda_{2},\Lambda_{3}\right\}  \cup  \left\{  \overline{\Lambda}\in \mathscr{L}\left(  W\right)  :\Lambda
_{1},\Lambda_{2},\Lambda_{3}\text{ are colinear in \, }\overline{\Lambda
}^{\pitchfork}\right\}
$$

\begin{proposition}\label{C1}
If  $\Lambda_{1},\Lambda_{2},\Lambda_{3}$ are GP concyclic,
then $C\left(  \Lambda_{1},\Lambda_{2},\Lambda_{3}\right)  $ is a 
regular cycle, and it is the only cycle that contains the three points.
\end{proposition}

\begin{proof}
If $\Lambda_{1},\Lambda_{2},\Lambda_{3}\in \mathscr{L}\left(  W\right)$ are GP concyclic, let $\overline{\Lambda}\in \mathscr{L}\left(  W\right)
$, and  $L$ an affine line in $\overline{\Lambda}^{\pitchfork}$ such that
$\Lambda_{1},\Lambda_{2},\Lambda_{3}\in L$.  Consider the cycle
$C=L\cup\left\{  \overline{\Lambda}\right\}$.   Necessarily the cycle $C$ is regular because $\Lambda_{1}\cap\Lambda_{2}=0$) and we get:

a) $L\subset$ $C\left(  \Lambda_{1},\Lambda_{2},\Lambda_{3}\right)$, because if $\Lambda_{0}\in L\backslash\left\{  \Lambda_{1},\Lambda_{2},\Lambda
_{3}\right\}  $, taking $S=\mathcal{S}_{\overline{\Lambda}^{\pitchfork}%
}^{\Lambda_{0}}$, and $S\left(  \Lambda_{i}\right)  =S_{i}$, $i=0,1,2,3$, with
$S_{0}=0$, and because $\Lambda_{0},\Lambda_{1},\Lambda_{2}$ are colinear in 
$\overline{\Lambda}^{\pitchfork}$, there exist $\lambda_{i}\neq0$ such that
$S_{i}=\lambda_{i}S_{1}$,  with $i=0,2,3$. 

Moreover, because  $\Lambda_{0}\cap\Lambda_{i}=0$ the matrices $S_{i}$ are invertible and if $\widetilde{S}=\mathcal{S}_{\Lambda_{0}^{\pitchfork}}^{\overline{\Lambda}}$, then
$\widetilde{S}=S^{-1}$ in $\overline{\Lambda}^{\pitchfork}\cap\Lambda
_{0}^{\pitchfork}$.   Calling  $\widetilde{S}\left(  \Lambda
_{i}\right)  =\widetilde{S}_{i}$, we get $\widetilde{S}_{i}=S_{i}^{-1}$.
Hence:
\[
S_{i}=\lambda_{i}S_{1}\Longrightarrow\widetilde{S}_{i}=S_{i}^{-1}=\frac
{1}{\lambda_{i}}S_{1}^{-1}=\frac{1}{\lambda_{i}}\widetilde{S}_{1}\text{,
}i=2,3 \, ,
\]
and we conclude that  $\Lambda_{1},\Lambda_{2},\Lambda_{3}$ are colinear in $\Lambda_{0}^{\pitchfork}.$

b) We prove now that  $C\left(  \Lambda_{1},\Lambda_{2},\Lambda_{3}\right)
\subset L\cup\left\{  \overline{\Lambda}\right\}  $.  Namely, if
$\Lambda_{0}\in C\left(  \Lambda_{1},\Lambda_{2},\Lambda_{3}\right)
\cap\overline{\Lambda}^{\pitchfork}$, in  Remark \ref{rem} it was shown that $L$ is also a line in $\overline{\Lambda}_{0}%
^{\pitchfork}$ if and only if $\Lambda_{0}\in L$. Thus because $\Lambda_{0}\in C\left(  \Lambda_{1},\Lambda_{2},\Lambda_{3}\right)  $, we conclude that  $\Lambda_{0}\in L.$

d) Now we will check that  $C\left(  \Lambda_{1},\Lambda_{2},\Lambda_{3}\right)  $ is the unique
cycle containing $\Lambda_{1},\Lambda_{2},\Lambda_{3}$.  If $C$ is a cycle containing $\Lambda_{1},\Lambda_{2},\Lambda_{3}$
we can write $C=L\cup\left\{  \overline{\Lambda}\right\}  $ with $L$ an affine line
in $\overline{\Lambda}^{\pitchfork}$ such that $\Lambda_{1},\Lambda_{2},\Lambda_{3}\in L$, and we just proved that the equality  $C=C\left(  \Lambda_{1},\Lambda_{2},\Lambda_{3}\right)$, is satisfied.
\end{proof}

\begin{remark}
If $\Lambda_{1},\Lambda_{2},\Lambda_{3}$ are concyclic but they are not in general position, then they lie in a singular line  $L$ in $\overline{\Lambda}^{\pitchfork}$ for some $\overline{\Lambda}\in \mathscr{L}\left(  W\right)$.
Then, $\Lambda_{i}\cap\Lambda_{j}=0$ for all $i,j$. 
In this situation $C=L\cup\left\{  \overline{\Lambda}\right\}$ is the only cycle that contains them, and it is natural to denote $C=C\left(
\Lambda_{1},\Lambda_{2},\Lambda_{3}\right)  $.
\end{remark}

\begin{corollary}
$\Lambda_{1},\Lambda_{2},\Lambda_{3}\in \mathscr{L}\left(  W\right)  $ are 
concyclic if and only if they lie in the same cycle.
\end{corollary}


\section{Conclusions and Discussion}

The structure of Jacobi curves, i.e., regular smooth curves on the Lagrangian Grassmannian of a symplectic vector space has been completely elucidated.   The construction of the Ricci curvature endomorphism together with the conformal geometric arc of the curve are instrumental to define the family of absolute conformal symplectic curvatures characterizing the curve.    Such construction relies on a new definition of the derivative curve of the Jacobi curve based on the properties of the local affine structure of the Lagrangian Grassmannian.   The natural extension of Cartan's theory of moving frames to the symplectic/Lagrangian setting together with a careful analysis of the Cartan matrix of the symplectic moving frame determined by a Jacobi curve allows to prove a reconstruction theorem for a Jacobi curve out of its conformal symplectic curvatures.

The theory presented here reproduces some fundamental traits in the treatment by Agrachev and Zelenko, like the Ricci curvature tensor of a Jacobi curve,  but following a different line of argument that we hope could clarify some of the geometrical content of the theory developed by these authors.   One of the significant contributions of the present work concerning the construction of curvature invariants is that the order of differentiability of the curvatures constructed according to the the theory presented here is substantially lower that those obtained by the aforementioned authors, thus, for instance, the normal element of arc in \cite{Ag02} is of order 5, while the conformal element of arc presented here is of order 3.  In addition to all this, the algorithmic computations leading to the construction of the curvatures detailed in Sect. \ref{sec:algorithm} can be implemented on any standard symbolic manipulation software.

An important outcome of the theory developed in this paper would be the construction of explicit curvature invariants characterizing solutions of Riccati equations that could be used to provide an alternative description of the phase space portrait of such and related equations.  To end this discussion, it is relevant to point out that Jacobi curves associated to null geodesics on Lorentzian manifolds provide a natural and relevant application of the theory.   The analysis of the new family of spacetime conformal invariants, called  conformal sky-invariants \cite{Ba22}, from the perspective offered by the results obtained in the present article will be the subject of subsequent work.


\section*{Acknowledgements}
The authors acknowledge financial support from the Spanish Ministry of Economy and Competitiveness, through the Severo Ochoa Programme for Centres of Excellence in RD (SEV-2015/0554), and the MINECO research project  PID2020-117477GB-I00.



\end{document}